\documentclass[12pt,leqno]{amsart}
\usepackage{graphicx}%
\usepackage{multirow}%
\usepackage{amsmath,amssymb,amsfonts,mathtools, esint}%
\usepackage{amsthm}%
\usepackage{mathrsfs}%
\usepackage[title]{appendix}%
\usepackage{xcolor}%
\usepackage{textcomp}%
\usepackage{manyfoot}%
\usepackage{booktabs}%
\usepackage{algorithm}%
\usepackage{algorithmicx}%
\usepackage{algpseudocode}%
\usepackage{listings}%
\usepackage{bbm}
\usepackage[pagebackref=true, colorlinks=true, citecolor=blue]{hyperref}
\theoremstyle{thmstyleone}%
\newtheorem{theorem}{Theorem}%
\newtheorem{remark}{Remark}%
\newtheorem{defn}{Definition}%
\newtheorem{thm}{Theorem}

\newtheorem{assm}{Assumption}[section]
\newtheorem{lem}[theorem]{Lemma}
\newtheorem{cor}[theorem]{Corollary}
\numberwithin{theorem}{section} 

\newcommand{\bx}{{\bf x}}

\newcommand{\R}{\mathbb{R}}

\usepackage[foot]{amsaddr}
\usepackage[margin=1in]{geometry}

\newcommand{\bn}{{\bf n}}
\newcommand{\bg}{{\bf g}}
\newcommand{\sU}{{\mathscr{U}}}
\newcommand{\s}{\color{black}}
\newcommand{\cred}{\color{red}}
\newcommand{\sB}{\mathcal B}
\newcommand{\sD}{\mathscr D}

\newcommand{\Tr}{\text{Tr}}
\newcommand{\bQ}{{\bf Q}}

\newcommand{\sC}{\mathcal C}

\newcommand{\sO}{\mathcal O}

\newcommand{\sS}{\mathcal S}

\newcommand{\sY}{\mathcal Y}

\newcommand{\ep}{\varepsilon}

\newcommand{\sL}{\mathscr{L}}
\newcommand{\sF}{\mathcal F}
\newcommand{\bu}{{\bf u}}
\newcommand{\bq}{{\bf q}}
\newcommand{\bP}{{\mathbb P}}
\newcommand{\bw}{{\bf w}}
\newcommand{\bW}{{\bf W}}
\newcommand{\br}{{\bf r}}

\newcommand{\bE}{{\mathbb E}}

\newcommand{\bD}{{\bf D}}
\newcommand{\bH}{{\bf H}}
\newcommand{\bC}{{\bf C}}

\newcommand{\bU}{{\bf U}}
\newcommand{\bL}{{\bf L}}
\newcommand{\sW}{\mathscr{W}}
\newcommand{\sV}{\mathscr{V}}

\newcommand{\bv}{{\bf v}}

\newcommand{\bfeta}{\boldsymbol{\eta}}
\newcommand{\bfpsi}{\boldsymbol{\psi}}
\newcommand{\bfphi}{\boldsymbol{\phi}}
\newcommand{\bfvarphi}{\boldsymbol{\varphi}}

\renewcommand{\tilde}{\widetilde}

\title[Stochastic moving boundary FSI problem]{A 2D stochastic fluid-structure interaction problem in compliant arteries with non-zero longitudinal displacement}
	\author[K. Tawri]{Krutika Tawri$^{1}$}

\address{\newline	$^1$ Department of Mathematics, University of California Berkeley, CA, USA.}

\email{ktawri@berkeley.edu }

\begin{document}

\maketitle

\begin{abstract}
In this paper, we study a nonlinear fluid-structure interaction problem driven by a multiplicative, white-in-time noise. The problem consists of the  Navier-Stokes equations describing the flow of an incompressible, viscous fluid in a 2D cylinder interacting with an elastic wall whose elastodynamics is described by a membrane/shell equation. The stochastic force is applied both to the fluid equations as a volumetric body force, and to the structure as an external forcing to the deformable fluid boundary. The fluid and the structure are nonlinearly coupled via the	kinematic and dynamic conditions assumed at the moving interface, which is a random variable not known a priori. In particular, we consider the case where the structure is allowed to have non-zero longitudinal displacement.\\

\end{abstract}

\section{Introduction}

This paper introduces a constructive approach for investigating martingale solutions to a nonlinear stochastic problem describing the interaction between a deformable elastic membrane and a two-dimensional viscous, incompressible fluid, both subject to multiplicative stochastic forces. The fluid equations are given by the 2D Navier-Stokes equations and the elastodynamics of the structure is given by membrane/ shell equations. The fluid and the structure are fully coupled across the  moving interface via a two-way coupling that ensures continuity of their velocities and contact forces at the fluid-structure interface. This two-way coupling gives rise to a strong geometric nonlinearity in the problem, since the location of the fluid domain is not known a priori and is one of the unknowns in
the problem.

The main result of this paper is the establishment of the existence of weak martingale solutions to this highly nonlinear stochastic fluid-structure interaction, moving boundary problem via an operator splitting scheme. The recent article \cite{TC23}, provided the existence of weak martingale solutions to a stochastic FSI problem where the structure displacement is assumed only in the radial (vertical) direction. In this paper we remove this restriction and consider a thin structure which is allowed to be displaced longitudinally. 
Due
to the non-zero longitudinal displacement of the structure, extra care has to be taken to deal with the fluid domains degeneracies, i.e. when the structure touches a part of the fluid domain boundary during its deformation.

There has been a lot of work done in the field of deterministic FSI in the past two decades (see e.g. \cite{CDEG, G08,KT12, MC13, GH16} and the references therein). However, even though there is evidence pointing to the need for studying their stochastic perturbations \cite{StochasticHeartFSI}, the mathematical theory of stochastic FSI or, more generally, of stochastic PDEs on randomly moving domains is undeveloped. 
Moreover, in most of the deterministic FSI literature, the structure is allowed to be displaced
only in the radial direction. To the best of our knowledge only \cite{MCG20} and \cite{KSS23} have studied the existence of weak solutions in the more general, unrestricted structure displacement case and, this present work is the {\bf first existence result} involving unrestricted structure displacement for a {\bf stochastic} moving boundary problem.

The coupled problem is discretized in time and split into a structure and a fluid subproblem using a Lie operator splitting scheme. 
The mathematical issues that we come across while building the splitting scheme are related to the following: (1) the fluid domain boundary is a random variable, not known a priori and can possibly degenerate in a random fashion and, (2)
the incompressibility condition and the kinematic coupling conditions require the test functions to be random. The first issue here is handled by introducing an appropriate cut-off function whereas the second issue is handled via a penalization method.

We use the Arbitrary Lagrangian Eulerian (ALE) mapping approach to transform the fluid equations from the time-dependent domain to a fixed domain.  The use of the ALE mappings and the analysis that follows is valid if there is no loss of injectivity of the ALE transformation.
While similar issues arise in the deterministic case as well, their resolution in the stochastic case is markably different. We can not use the ALE maps introduced in \cite{MCG20} since they would further enforce the dependence of the test functions on the domain configurations. Unlike the incompressibility, we can not penalize the boundary behavior of the fluid and structure velocities either since we rely on obtaining partial regularity for the structure velocity, as the trace of the fluid velocity on the interface, from the fluid viscous dissipation. This is required to obtain convergence of the structure velocity in strong topology to pass to the limit in the weak formulation, in particular in the stochastic integral.
Moreover, handling injectivity of the ALE maps in the stochastic case requires a different approach, which in this manuscript is done using a cut-off function and a stopping time argument. The cut-off function artificially provides a {\it deterministic} lower bound on the Jacobian of the ALE maps and a {\it deterministic} upper bound on the $H^s$-norm (large enough $s<2$) of the structure displacement. Using this cut-off function, artificial structure displacement variables are introduced in a way that still provides us with a stable scheme. A stopping time argument is then developed to 
show that this cut-off function does not ``kick in'' until some stopping time which is strictly positive almost surely.

Furthermore, the incompressibility condition requires that we construct test functions that are random due to their dependence on the domain configurations via the ALE mappings. This creates challenges as we apply, for example, the Skorohod representation theorem and move to a new probability space to upgrade to almost sure convergence of the approximate solutions. Hence we introduce a system that approximates the original system by augmenting it by a singular term that penalizes the divergence of the approximate solutions. However, addition of this penalty term and the low temporal regularity of the solutions create further difficulties in establishing compactness which we overcome by employing non-standard compactness arguments in Section \ref{sec:tight}.  We then show that the solutions to the approximate systems indeed converge to a desired martingale solution of the limiting equations.

Finally, since the stochastic forcing appears not only in the structure equations but also in the fluid equations themselves, we come across additional difficulties, which are associated with the construction of the 
appropriate "test processes" on the approximate and limiting moving fluid domains.  Namely, along with the required divergence-free property on these domains, the test functions also have to satisfy appropriate boundary conditions and measurability properties. We construct these approximate test functions by first constructing a Carath\'eodory function that gives the definition of a test function for the limiting equations
and then by transforming it on the approximate domains in a way such that the desired properties are preserved.

\section{Problem setup}\label{sec:det_setup}
We begin describing the problem by first considering the deterministic model.

	\subsection{The deterministic model and a weak formulation}
We consider the flow of an incompressible, viscous fluid in a two-dimensional compliant cylinder $\sO=(0,L)\times(0,1)$ with a deformable lateral boundary $\Gamma$.
	The left and the right boundary of the cylinder are the inlet  and outlet for the time-dependent fluid flow. 
	We assume ``axial symmetry'' of the data and of the flow, rendering the central horizontal line as the axis of symmetry.	This allows us to  consider the flow only in the upper half of the domain, with the bottom boundary fixed and equipped with the symmetry boundary conditions. 

Assume that the time-dependent fluid domain, whose displacement is not known a priori is denoted by 	$\sO_{{\bfeta}}(t)={\bfvarphi}(t,\sO)$ whereas its deformable interface is given by $\Gamma_{\bfeta}(t)={\bfvarphi}(t,\Gamma)$. Assume that ${\bfvarphi}:\sO\rightarrow\sO_{{\bfeta}}$ is a $C^1$ diffeomorphism such that
$${\bfvarphi}|_{\Gamma_{in},\Gamma_{out},\Gamma_{b}}=id,\quad det\nabla{\bfvarphi}( t,\bx)>0,$$
where the left, right and bottom boundaries of $\sO$ are given by $\Gamma_{\text{in}}=\{0\}\times (0,1),\Gamma_{\text{out}}=\{L\}\times (0,1),\Gamma_b=(0,L)\times \{0\}$ respectively.
The displacement of the elastic structure at the top lateral boundary $\Gamma$, which can be identified by $(0,L)$, will be given by ${\bfeta}(t,z)={\bfvarphi}(t,z)-(z,1)$ for $z\in (0,L)$ (see Fig. \ref{fig}). The mapping ${{\bfeta}}:[0,L]\times[0,T] \rightarrow \R^2$ such that ${\bfeta}=({\eta}_z(z,t),{\eta}_r(z,t))$ is one of the unknowns in the problem.

\begin{figure}[h]\centering	\includegraphics[scale=0.8]{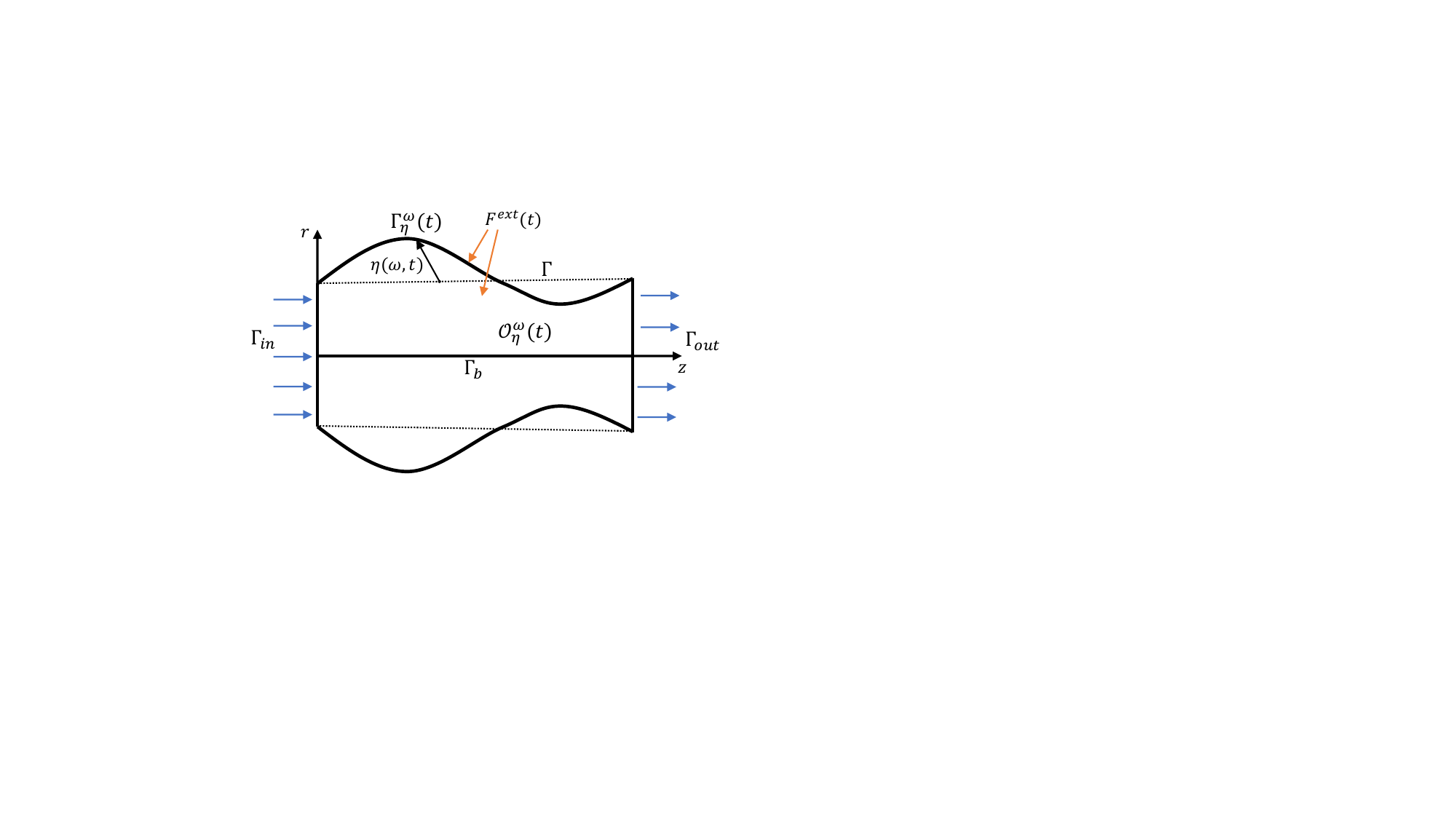}
	\caption{A realization of the domain}\label{fig}
\end{figure}

\noindent{\bf The fluid subproblem:}		
The fluid flow is modeled by the incompressible Navier-Stokes equations in the 2D time-dependent domain $\sO_{{\bfeta}}(t)$:  
\begin{equation}
	\begin{split}\label{u}
		\partial_t\bu + (\bu\cdot \nabla)\bu &= \nabla \cdot \sigma + F^{ext}_u\\
		\nabla \cdot \bu&=0,
	\end{split}
\end{equation}
where $\bu=(u_z,u_r)$ is the fluid velocity. Hereon, subscripts with $z$ and $r$ will denote the first and the second components respectively. The Cauchy stress tensor is $\sigma=-p I+2\nu \bD(\bu)$ where $p$ is the fluid pressure, $\nu$ is the kinematic viscosity coefficient and $\bD(\bu)=\frac12(\nabla\bu+(\nabla\bu)^T)$ is the symmetrized gradient. Here $F_u^{ext}$ represents any external forcing impacting the fluid.	{\s In this work we will be assuming that this force is random, as we shall see below.} The fluid flow is driven by dynamic pressure data given at the inlet and the outlet boundaries as follows:
\begin{equation}
	\begin{split}\label{bc:in-out}
		p +\frac12 |\bu|^2&=P_{in/out}(t),\\
		u_r&=0 \quad\text{ on } \Gamma_{in/out}.
\end{split}	\end{equation}
Whereas on the bottom boundary $\Gamma_b$ we prescribe the symmetry boundary condition:
\begin{align}
	u_r=\partial_r u_z=0 \quad\text {on } \Gamma_b.\label{bc:bottom}
\end{align}		 
{\bf The structure sub-problem:} The elastodynamics problem is given in terms of the displacement
${\bfeta}=(\eta_z,\eta_r)$ with respect to $\Gamma$:
\begin{align}\label{eta}
	\partial^2_{t}{\bfeta} +\sL_e({\bfeta}) = F_{\bfeta} \quad \text{ in } (0,L),
\end{align}
where $F_{{\bfeta}}$ is the total force experienced by the structure and $\sL_e$ is a continuous, self-adjoint, coercive, linear operator on $\bH^2_0(0,L)$.	
The above equation is supplemented with the following boundary conditions:
\begin{equation}\label{bc:eta}
	\begin{split}
		{\bfeta}(0)={\bfeta}(L)=\partial_z{\bfeta}(0)=\partial_z{\bfeta}(L)=0.
	\end{split}
\end{equation}
{\bf The non-linear fluid-structure coupling}: The coupling between the fluid and the structure takes place {\s across the current location of the fluid-structure interface, which is simply the current location of the membrane/shell, described above. 
\begin{itemize}
	\item The kinematic coupling conditions are:
	\begin{align}\label{kinbc1}
		\partial_t {\bfeta}(t,z) =\bu({\bfvarphi}(t,z)),\quad (t,z)\in[0,T]\times [0,L].
	\end{align}
	\item
	The dynamic coupling condition is:
	\begin{align}
		F_{\bfeta}=-S_{\bfeta}(t,z) (\sigma {\bf n}^{\bfeta})|_{(t,z,{\bfvarphi}(t,z))} + F_{\bfeta}^{ext},
	\end{align}
	where $\bn^{\bfeta}(t,z)$ is the unit outward normal to the top boundary at the point, and $S_{\bfeta}(t,z)$ is the Jacobian of the transformation from Eulerian to Lagrangian coordinates. As earlier, $F^{ext}_{\bfeta}$ denotes any external force impacting the structure.
\end{itemize}
This system is supplemented with the following initial conditions:
\begin{align}\label{ic}
	\bu(t=0)=\bu_0,{\bfeta}(t=0)={\bfeta}_0, \partial_t{\bfeta}(t=0)=\bv_0.
\end{align}

\subsection*{Weak formulation on moving domain}
Before we derive the weak formulation of the deterministic system described in the previous sub-section, we define the following relevant function spaces
{\s for the fluid velocity, the structure, and the coupled FSI problem}:
\begin{align*}
	&\tilde\sV_F(t)= \{{\bf u}=(u_z,u_r)\in \bH^{1}(\sO_{{\bfeta}}(t)) :
	\nabla \cdot {\bf u}=0,  u_r=0
	\text{ on }\partial \sO_{{\bfeta}}\setminus \Gamma_{{\bfeta}}(t)\}\nonumber\\
	&\tilde\sW_F(0,T)=L^\infty(0,T;\bL^2(\sO_{\bfeta}(\cdot))) \cap L^2(0,T;\tilde\sV_F(\cdot))\\
	&\tilde\sW_S(0,T)=W^{1,\infty}(0,T;\bL^2(0,L)) \cap L^\infty(0,T;\bH^2_0(0,L))\\
	&\tilde\sW(0,T)=\{(\bu,{\bfeta})\in \tilde\sW_F(0,T)\times\tilde\sW_S(0,T):{ \bu({\bfvarphi}(t,z))}=\partial_t{\bfeta}(t,z), (t,z)\in (0,T)\times\Gamma\}.
\end{align*} 
We use the convention that bold-faced letters denote spaces containing vector-valued functions.
{\s Next, we derive} a weak formulation of the problem on the moving domains. We consider $\bq \in C^1([0,T];\tilde\sV_F(\cdot))$ such that ${\bq}({\bfvarphi}(t,z))=\bfpsi(t,z)$ on $(0,T)\times \Gamma_{}$ for some $\bfpsi \in C^1([0,T];\bH^2_0(\Gamma))$. We multiply \eqref{u} by $\bq$, integrate in time and space and use Reynold's transport theorem to obtain,
\begin{align*}
	&(\bu(t), \bq(t))_{\sO_{{\bfeta}}(t)}= (\bu(0), \bq(0))_{\sO_{{\bfeta}}(0)} +\int_0^t\int_{\sO_{{\bfeta}}(s)} \bu(s)\cdot\partial_s\bq(s) \\
	&+ \int_0^t\int_{\Gamma_{{\bfeta}}(s)}(\bu(s)\cdot\bq(s))(\bu(s)\cdot \bn^{\bfeta}(s))
	-\int_0^t\int_{\sO_{{\bfeta}}(s)}(\bu(s)\cdot \nabla )\bu(s) \bq(s)\\
	&- 2\nu\int_0^t \int_{\sO_{{\bfeta}}(s)}  \bD(\bu(s))\cdot\bD(\bq(s))  ds+\int_0^t\int_{\partial\sO_{{\bfeta}}(s)} (\sigma \bn^{\bfeta}(s))\cdot \bq(s)+\int_0^t\int_{\sO_{{\bfeta}}(s)}F^{ext}_u(s)\bq(s).
\end{align*}
Let,
$$b(t,\bu,{\bf v},{\bf w}):=\frac12\int_{\sO_{{\bfeta}}(t)}\left( (\bu\cdot\nabla){\bf v}\cdot{\bf w}-(\bu\cdot\nabla){\bf w}\cdot{\bf v}\right)  .$$
With a slight abuse of notation, we let $\bn^{\bfeta}$ denote the unit normal to $\partial\sO_{\bfeta}$ in the following calculation:
\begin{align*}
	-((\bu\cdot\nabla)\bu,\bq)_{\sO_{{\bfeta}}}&=-\frac12((\bu\cdot\nabla)\bu,\bq)_{\sO_{{\bfeta}}}+\frac12((\bu\cdot\nabla)\bq,\bu)_{\sO_{{\bfeta}}} -\frac12\int_{\partial\sO_{{\bfeta}}}\bu\cdot\bq \bu\cdot \bn^{\bfeta} \\
	&=-b(s,\bu,\bu,\bq) -\frac12\int_{\Gamma_{{\bfeta}}}\bu\cdot\bq\bu\cdot \bn^{\bfeta}+ \frac12\int_{\Gamma_{in}}|u_z|^2q_z-\frac12\int_{\Gamma_{out}}|u_z|^2q_z.
\end{align*}
Next we consider the structure equation \eqref{eta}. We multiply \eqref{eta} by $\bfpsi$ and integrate in time and space to obtain
\begin{align*}
	(\partial_t{\bfeta}(t),\bfpsi(t))&=(\bv_0,\bfpsi(0) +\int_0^t\int_0^L\partial_s{\bfeta}\cdot\partial_s\bfpsi dzds - \int_0^t\langle \sL_e({\bfeta}),\bfpsi \rangle ds  \\
	&-\int_0^t\int_0^LS_{\bfeta}\sigma \bn^{\bfeta} \cdot\bfpsi dzds +\int_0^t \int_0^LF_{\bfeta}^{ext}\cdot\bfpsi dz ds.
\end{align*}
Hence, to summarize, for any test function $\bQ=(\bq,\bfpsi)$ we look for a solution $(\bu,{\bfeta}) \in \tilde\sW(0,T)$,  that satisfies the following equation for {almost} every $t \in [0,T]$:
\begin{equation}
	\begin{split}\label{origweakform}
		&{\int_{\sO_{{\bfeta}(t )}}\bu(t )\bq(t ) d\bx+\int_0^L\partial_t{\bfeta}(t)\bfpsi(t )dz}-\int_0^{t }\int_{\sO_{{\bfeta}(s)}}\bu\cdot\partial_s\bq d\bx ds\\
		&+\int_0^{t} b(s,\bu,\bu,\bq)ds -\frac12\int_0^{t}\int_{\Gamma_{\bfeta}}(\bu\cdot\bq)(\bu\cdot\bn^{\bfeta}) dSds+ 2\nu\int_0^{t} \int_{\sO_{{\bfeta}(s)}} \bD(\bu)\cdot \bD(\bq) d\bx ds\\
			&-\int_0^{t}\int_0^L\partial_s{\bfeta}\partial_s\bfpsi dzds +\int_0^{t } \langle\sL_e({\bfeta}),\bfpsi \rangle ds
		=\int_{\sO_{{\bfeta}_0}}\bu_0\bq(0) d\bx+ \int_{0}^L \bv_0\bfpsi(0) dz\\
&	+\int_0^{t}P_{{in}}\int_{0}^1q_z\Big|_{z=0}drds-\int_0^tP_{{out}}\int_{0}^1q_z\Big|_{z=1}drds \\
&	 +\int_0^{t}\int_{\sO_{{\bfeta}}(s)}\bq \cdot  F_{u}^{ext}\, d\bx ds 
		+\int_0^{t}\int_0^L \bfpsi\cdot  F_{{\bfeta}}^{ext}\, dz ds,
\end{split}\end{equation}
where $F_{u}^{ext}$ is the volumetric external force and $F_{{\bfeta}}^{ext}$ is the external force applied to the deformable boundary, 
chosen to be random in nature, as we explain next. 
\subsection{The stochastic framework and definition of martingale solutions}
We will assume that the external forces $F_u^{ext}$ and ${\s F_{{\bfeta}}^{ext}}$ are multiplicative stochastic forces and  that we can then write the combined stochastic forcing $F^{ext}$ as follows:
\begin{equation}\label{StochasticForcing}
	{\s F^{ext}} := G(\bu,\bv,{\bfeta}) {dW},
\end{equation}
where $\bu$ is the fluid velocity, $\bv$ is the structure velocity, ${\bfeta}$ is the structure displacement, and $W$ is a Wiener process. 

We will now give the relevant probabilistic framework. 	The stochastic noise term is
defined on a given filtered probability space $(\Omega,\sF,(\sF_t)_{t \geq 0},\bP)$ that satisfies the usual
assumptions, i.e., $\sF_0$ is complete and the filtration is right continuous, that is, $\sF_{t}=\cap_{s \geq t}\sF_s$ for all $t \geq 0$. 
We assume that $W$ is a $U$-valued Wiener process with respect to the filtration $(\sF_t)_{t \geq 0}$, where $U$ is a separable Hilbert space. We denote by $Q$ 
the covariance operator of $W$, which is a positive, trace class operator on $U$, 
and define $U_0:=Q^{\frac12}(U)$.

Introducing the notation $\bL^2:=\bL^2(\sO)\times \bL^2(0,L)$ we now give assumptions on the noise coefficient $G$:
{\s \begin{assm}\label{G}
		Let $L_2(X,Y)$ denote the space of Hilbert-Schmidt operators from a Hilbert space $X$ to another Hilbert space $Y$. The {\bf noise coefficient} $G$ is a function $G:\sU\times \bH^s(0,L) \rightarrow L_2(U_0;\bL^2)$, for any $\frac32<s<2$, such that the following assumptions hold:
		\begin{equation}\begin{split}\label{growthG}
				&\|G(\bu,\bv,{\bfeta})\|_{L_2(U_0;\bL^2)} \leq \|{\bfeta}\|_{\bH^s(0,L)}\|{\bu}\|_{\bL^2(\sO)} + \|\bv\|_{\bL^2(0,L)},\\
				&\|G(\bu_1,\bv_1,{\bfeta})-G(\bu_2,\bv_2,{\bfeta})\|_{L_2(U_0;\bL^2)} \leq \|{\bfeta}\|_{\bH^s(0,L)}\|{\bu_1}-{\bu_2 }\|_{\bL^2(\sO)} + \|\bv_1-\bv_2\|_{\bL^2(0,L)},\\
				&\|G(\bu,\bv,{\bfeta}_1)-G(\bu,\bv,{\bfeta}_2)\|_{L_2(U_0;\bL^2)} \leq \|{\bfeta}_1-{\bfeta}_2\|_{\bH^s(0,L)}\|\bu\|_{\bL^2(\sO)}.
		\end{split}\end{equation}	
	\end{assm}
	
	We will use this framework below to define a martingale solution to our stochastic FSI problem.
	In order to do this, we first transform the problem defined on the moving domains onto a fixed reference domain using
	a family of Arbitrary Lagrangian-Eulerian (ALE) mappings. The mappings are defined next.
	
	\subsubsection{ALE mappings}\label{sec:ale}
	To deal the geometric nonlinearity arising due to the motion of the fluid domain, we consider the arbitrary Lagrangian-Eulerian (ALE) mappings which are a family of diffeomorphisms from the fixed domain $\sO=(0,L)\times(0,1)$ onto {the moving domain}  $
	\sO_{{\bfeta}}(t)$. Notice that the presence of the stochastic forcing implies that the domains $\sO_{\bfeta}$ are themselves random and that we must consider the ALE mappings defined {\it pathwise}, that is, for every $\omega\in \Omega$ we will consider the maps $
	A_{{\bfeta}}^\omega(t):\sO \rightarrow \sO_{{\bfeta}}(t,\omega)
	$ such that $A^\omega_{\bfeta}(t)=\textbf{id}+{\bfeta}(t,\omega) \text{ on } \Gamma$ and $
	A^\omega_{\bfeta}(t)=\textbf{id}  \text{ on } \partial\sO\setminus\Gamma.$

	Assuming the existence of such a map we will give the definition of martingale solutions (see Definition \ref{def:martingale}) on fixed domain by first introducing the required notations.
	Under this transformation, the pathwise transformed gradient and symmetrized gradient will be denoted by
	$$\nabla^{\bfeta}f=\nabla f(\nabla A_{\bfeta})^{-1},  \text{ and } \bD^{\bfeta}(\bu)=\frac12(\nabla^{\bfeta}\bu+(\nabla^{\bfeta})^T\bu). $$
	This gives us the definition of the transformed divergence as:
	$$	div^{\bfeta}f=tr(\nabla^{\bfeta}f).$$	
	The Jacobian of the ALE mapping is given by
	\begin{equation}\label{ALE_Jacobian}
		J^\omega_{\bfeta}(t)=\text{det }\nabla A_{{\bfeta}}^\omega(t).
	\end{equation}
	Using $\bw^{\bfeta}$ to denote the ALE velocity
	$\bw^{\bfeta}=\frac{d}{dt}A_{\bfeta}$, 
	we rewrite the advection term as follows:
	$$b^\eta(\bu,\bw^{\bfeta},\bq)=\frac12\int_{\sO}J_{{\bfeta}}\left(((\bu-\bw^{\bfeta})\cdot\nabla^{\bfeta} )\bu\cdot\bq-((\bu-\bw^{\bfeta})\cdot\nabla^{\bfeta} )\bq\cdot\bu\right). $$
	\subsubsection{Definition of martingale solutions} 
	{\s We start by introducing} the {\bf functional framework for the stochastic problem on the fixed reference domain} $\sO=(0,L) \times (0,1)$.  The following are the functions spaces for the stochastic FSI problem defined on the fixed domain $\sO$:
	\begin{align*}
		& V= \{{\bf u}=(u_z,u_r)\in {\bH}^{1}(\sO): 
	u_r=0 \text{ on }\partial \sO_{}\setminus \Gamma \},\\
		&\sW_F=L^2(\Omega;L^\infty(0,T;\bL^2(\sO))) \cap L^2(\Omega;L^2(0,T;V)),\\
		&\sW_S=L^2(\Omega;W^{1,\infty}(0,T;\bL^2(0,L)) \cap L^\infty(0,T;\bH^2_0(0,L))),\\
		&\sW(0,T)=\{(\bu,{\bfeta})\in \sW_F\times \sW_S:
		\bu(t)|_{\Gamma}=\partial_t{\bfeta}(t),\,\,\nabla^{\bfeta} \cdot {\bu}=0\, \,\bP-a.s.\}.
	\end{align*} 
	We define the following spaces for test functions and for fluid and structure velocities:
	\begin{equation}\label{Dspace}
		\sD=\{({\bf q},\bfpsi) \in  V \times \bH^2_0(0,L): {\bq}|_\Gamma=\bfpsi
		\},\,
		\sU=\{(\bu,\bv)\in V\times \bL^2(0,L): \bu|_{\Gamma}=\bv \}.
	\end{equation}
	
	\begin{defn}({\bf \s Martingale solution})\label{def:martingale}
		Given compatible deterministic initial fluid and structure velocity data,  $\bu_0 \in \bL^2(\sO)$, $\bv_0\in \bL^2(0,L)$ and compatible initial structure displacement ${\bfeta}_0\in \bH_0^2(0,L)$ 
			such that
			for some $\delta_1,\delta_2>0$,
			\begin{align}	 \label{etainitial}
				{\delta_1<\inf_{\sO}J_{{\bfeta}_0},\quad \text{ and }\quad \|{\bfeta}_0\|_{\bH^2_0(0,L)}<\frac1{\delta_2}},
		\end{align}
		we say that  $(\mathscr{S},\bu,{\bfeta},T^{\bfeta})$ is a  martingale solution to the system \eqref{u}-\eqref{ic} under the assumptions \eqref{growthG} if 
		\begin{enumerate}
			\item  $\mathscr{S}=(\Omega,\sF,(\sF_t)_{t\geq 0},\bP,W)$ is a stochastic basis, that is, $(\Omega,\sF,(\sF_t)_{t\geq 0},\bP)$ is a filtered probability space satisfying the usual conditions and $W$ is a $U$-valued Wiener process.
			\item $(\bu,{\bfeta})\in \sW(0,T)$. 
			\item $T^{\bfeta}$ is a $\bP$-a.s. strictly positive, $\sF_t-$stopping time.
			\item  
			$(\bu,\partial_t{\bfeta})$ and ${\bfeta}$ are  $(\sF_t)_{t \geq 0}-$progressively measurable.   
			\item	For every $(\sF_t)_{t \geq 0}-$adapted, essentially bounded process $\bQ:=(\bq,\bfpsi)$ with $C^1$ paths in $\sD$ such that 
			$\nabla^{\bfeta} \cdot \bq=0$, the equation 
		\end{enumerate}
		\begin{equation}\label{weaksol}
			\begin{split}
				&{\int_{\sO}J_{\bfeta}(t)\bu(t)\bq(t) +\int_0^L\partial_t{\bfeta}(t)\bfpsi(t)}= \int_{\sO}J_0\bu_0\bq(0)   + \int_{0}^L \bv_0\bfpsi(0)  \\
				&+\int_0^{t }\int_{\sO}J_{\bfeta}\,\bu\cdot \partial_t\bq -{\int_0^{t }b^{{\bfeta}}(\bu,\bw^{\bfeta},\bq)}+\int_0^t\int_\sO J_{\bfeta}(\nabla^{\bfeta}\cdot\bw)\bu\cdot\bq\\
				&- 2\nu\int_0^{t } \int_{\sO} J_{\bfeta}\bD^{\bfeta}(\bu)\cdot \bD^{\bfeta}(\bq)
				+\int_0^{t }\int_0^L\partial_t{\bfeta}\partial_t\bfpsi -\int_0^{t }\int_0^L \langle\sL_e({\bfeta}),\bfpsi \rangle \\
				&+\int_0^{t }\left( P_{{in}}\int_{0}^1q_z\Big|_{z=0}dr-P_{{out}}\int_{0}^1q_z\Big|_{z=1}dr\right) 
				+\int_0^{t }(\bQ,G(\bu,\partial_t\bfeta,{\bfeta})\,dW),
		\end{split}\end{equation}
		holds $\bP$-a.s. for almost every $t \in[0,T^{\bfeta})$.
	\end{defn}
	In the rest of this manuscript we will present a constructive proof of the existence of martingale solutions. The construction is based on the following operator splitting scheme.
	\section{Operator splitting scheme} \label{sec:splitscheme}
	In this section we introduce a Lie operator splitting scheme that defines a sequence of approximate solutions to \eqref{weaksol} by semi-discretizing the problem in time. The final goal is to show that up to a subsequence, approximate solutions converge in a certain sense to a martingale solution of the stochastic FSI problem. 
	
	\subsection{Definition of the splitting scheme} We semidiscretize the problem in time and use operator splitting to divide the coupled stochastic problem into two subproblems, a fluid and a structure subproblem.
	We denote the time step by $\Delta t=\frac{T}{N}$ and use the notation $t^n=n\Delta t$ for $n=0,1,...,N$. 
	
	Let $(\bu^0,{\bfeta}^0,\bv^0)=(\bu_0,{\bfeta}_0,\bv_0)$ be the initial data. At the $i^{th}$ time level, we update the vector $(\bu^{n+\frac{i}{2}},{\bfeta}^{n+\frac{i}{2}},\bv^{n+\frac{i}{2}})$, where $i=1,2$ and $n=0,1,2,...,N-1$, according to the following scheme.
	\subsection*{The structure sub-problem} 
	In this sub-problem we update the structure displacement and the structure velocity while keeping the fluid velocity the same. That is, given $({\bfeta}^n,\bv^n) \in \bH^2_0(0,L)\times \bL^2(0,L)$ we look for a pair $({\bfeta}^{n+\frac12},\bv^{n+\frac12}) \in \bH^2_0(0,L) \times \bH^2_0(0,L)$ that satisfies the following equations pathwise i.e. for each $\omega\in\Omega$:
	\begin{equation}
		\begin{split}\label{first}
			\bu^{n+\frac12}&=\bu^n,\\
			\int_0^L({\bfeta}^{n+\frac12}-{\bfeta}^n) \bfphi dz&= (\Delta t)\int_0^L \bv^{n+\frac12}\bfphi dz,\\
			\int_0^L \left( \bv^{n+\frac12}-\bv^n\right)  \bfpsi dz &+ (\Delta t)\langle \sL_e({\bfeta}^{n+\frac12}),\bfpsi \rangle+\ep(\Delta t)\int_0^L\partial_z\bv^{n+\frac12}\partial_z\bfpsi dz=0,
		\end{split}
	\end{equation}
	for any $\bfphi \in \bL^2(0,L)$ and $\bfpsi \in \bH^2_0(0,L)$.
	Note that the extra regularization term in this subproblem is added to provide the structure velocities $\bv^{n+\frac12}$ with the required regularity in order to establish compactness in Section \ref{sec:tight}. We will later pass $\ep \to 0$ by using the estimates on the fluid dissipation and the interplay between the
	convective term and the boundary deformation.

Before commenting on the existence of the random variables ${\bfeta}^{n+\frac12},\bv^{n+\frac12}$ and their measurability properties, we introduce the associated ALE maps and the second sub-problem.

For each $\omega \in \Omega$ and $n$, we define the (pathwise) ALE map associated with this structure variable as the solution to,
\begin{equation}\label{ale}
	\begin{split}
		\Delta A^\omega_{{\bfeta}^n}&=0, \quad \text{ in } \sO,\\
		A^\omega_{{\bfeta}^n}&=\textbf{id}+{\bfeta}^n(\omega), \quad \text{ on } \Gamma,\\
		A^\omega_{{\bfeta}^n}&=\textbf{id}, \quad \text{ on } \partial\sO\setminus\Gamma.
	\end{split}
\end{equation}

	\subsection*{The fluid sub-problem}   In this step we update the fluid and the structure velocities while keeping the structure displacement unchanged.  However there are two major difficulties associated with this method in this sub-problem. The {\bf first difficulty} is associated with the fact that the fluid domains  can degenerate randomly. To overcome this problem we introduce an "artificial structure displacement" random variable by the means of a cut-off function as follows:
For $\delta=(\delta_1,\delta_2)$, let $\Theta_\delta$ be the
step function such that
$\Theta_{\delta}(x,y)=1$ if $\delta_1 < x, \text{} y<  \frac1{\delta_2}$, 
and $\Theta_{\delta}(x,y)=0$ otherwise. 
For brevity we define, for a fixed $s \in (\frac{3}2,2)$, a real-valued function $\theta_\delta({\bfeta}^n)$ which tracks all the structure displacements until the time step $n$, and is equal to 1 until the step for which the structure quantities leave the desired bounds given in terms of $\delta$:
\begin{align}\label{theta}
	\theta_\delta({\bfeta}^n):=\min_{k\leq n}\,
	\Theta_{\delta}\left( \inf_{\sO}J^k, \|{\bfeta}^k\|_{\bH^s(\Gamma)}\right),\end{align}
where $J^k(\omega)=\text{det}\nabla A^\omega_{{\bfeta}^k}$ is the Jacobian of the map defined in \eqref{ale}.

Now we define the artificial structure displacement random variable as follows,
\begin{align}\label{eta*}
	{\bfeta}^{n}_*(z,\omega)={\bfeta}^{\max_{0\leq k\leq n}\theta_\delta({\bfeta}^k)k}(z,\omega)\quad \text{for every } \omega \in \Omega,\end{align}
where the superscript in the definition above indicates the time step and not the power of structure displacement. This artificial structure variable is defined with care in a way that ensures the stability of our scheme.

Observe that, for any $p>2$ such that $s\geq\frac{5}2-\frac2p$, we have the following regularity result for the harmonic extension of the boundary data associated with $\bfeta^n_*$ on a square (see Section 5 in \cite{G11}):
\begin{align}\label{boundsA1}
	\|A^\omega_{{\bfeta}_*^n}-{\bf id}\|_{\bW^{2,p}(\sO)} \leq C\|{\bfeta}_*^n\|_{\bW^{2-\frac1{p},p}(\Gamma)} \leq C\|{\bfeta}_*^n\|_{\bH^s(\Gamma)} .
\end{align}
Hence, by Morrey's inequality, for some $C_*>0$ we obtain (see Theorem 7.26 in \cite{GT83}) for some $p<4$,
\begin{align}\label{boundJ}
	\|\nabla (A^\omega_{{\bfeta}_*^n}-{\bf id})\|_{\bC^{0,\frac1p}( \bar\sO)} \leq C\|{\bfeta}_*^n\|_{\bH^s(\Gamma)} \leq \frac{C_*}{\delta_2}, \quad \frac{5}2-\frac2p\leq s<2.
\end{align}
Now thanks to Theorem 5.5-1 (B) of \cite{C88}, if $\delta_2$ satisfies 
\begin{align}\label{delta}
	C_* < \delta_2,
\end{align}
then the map $A^{\omega}_{{\bfeta}^n_*}\in \bC^{1,\frac1p}(\bar\sO)$ is injective for any $n$. 

Hence the domain configurations corresponding the the artificial variables $\bfeta^n_*$ are non-degenerate and their Jacobians have a deterministic lower bound of $\delta_1$. We will use these artificial domain configurations to define the fluid sub-problem.

 We will define the fluid-subproblem on these artificial domain configurations.

The {\bf second difficulty} in defining the fluid sub-problem is associated with the fact that, through the transformed divergence-free condition, the test functions depend on the structure displacement found in the previous sub-problem.  However it is not clear how to deal with such test functions, as in the following section we construct and work on a new probability space.
Hence, in this sub-problem we supplement the weak formulation by a penalty term of the form $\frac1{\ep}\int \text{div}^{\bfeta} \bu \text{div}^{\bfeta}\bq$  {\s to enforce the incompressibility condition only in the limit as $\ep \rightarrow 0$}. 
	
	\noindent{\bf A penalized system:} We introduce a  penalty term and define the fluid subproblem as follows. Let $\Delta_n W:=W(t^{n+1})-W(t^n)$.\\
	Then for any $\ep>0$ and given $\bU^n=(\bu^n,\bv^n) \in \sU$ and $\delta_2$ that satisfies \eqref{delta}, we look for $ (\bu^{n+1},\bv^{n+1})\in \sU$ that solves
	\begin{equation}
		\begin{split}\label{second}
			&\qquad\qquad{\bfeta}^{n+1}:={\bfeta}^{n+\frac12}, \\
			&\int_{\sO}J^n_*\left( \bu^{n+1}-\bu^{n{}}\right) \bq d\bx  +\frac{1}2\int_\sO 
			\left( J_*^{n+1}-J^n_*\right)
	\bu^{n+1}\cdot\bq 
			d\bx +\int_0^L(\bv^{n+1}-\bv^{n+\frac12} )\bfpsi dz  \\
			&+\frac12(\Delta t)\int_{\sO}J_*^n((\bu^{n+1}-
			\bw^{{n}}_*)\cdot\nabla^{{\bfeta}_*^n}\bu^{n+1}\cdot\bq - (\bu^{n+1}-
			\bw^{{n}}_*)\cdot\nabla^{{\bfeta}_*^n}\bq\cdot\bu^{n+1})d\bx \\
			&+2\nu(\Delta t)\int_{\sO}J^n_* \bD^{{\bfeta}_*^{n}}(\bu^{n+1})\cdot \bD^{{\bfeta}_*^{n}}(\bq) d\bx 
			+ \frac{(\Delta t)}{\ep}\int_\sO 
			\text{div}^{{\bfeta}^n_*}\bu^{n+1}\text{div}^{{\bfeta}^n_*}\bq d\bx \\
	&=(\Delta t)\left( P^n_{{in}}\int_{0}^1q_z\Big|_{z=0}dr-P^n_{{out}}\int_{0}^1q_z\Big|_{z=1}dr\right) +  		(G(\bU^{n},{\bfeta}_*^n)\Delta_n W, \bQ),
		\end{split}
	\end{equation}
	for any $\bQ=(\bq,\bfpsi) \in \sU$.
	Here,
		$$P^n_{in/out}:=\frac1{\Delta t}\int_{t^n}^{t^{n+1}}P_{in/out}\,dt,\quad \text{div}^{{\bfeta}}\bu=tr(\nabla^{\bfeta}\bu),	$$ and
	$$\bw^{n}_*=\frac{1}{\Delta t}(A^\omega_{{\bfeta}^{n+1}_*}-A^\omega_{{\bfeta}^n_*}),\quad J^n_*=\text{det}\nabla A^\omega_{{\bfeta}^n_*}.$$
	\begin{remark}	
	Observe that using the data from the second subproblem, we update $\bfeta^n$ in the first subproblem and not $\bfeta_*^n$. This is required to obtain a stable scheme. However this also means that after a certain time (which will be random in nature), we will produce solutions that are meaningless.	However the discrepensies caused by the artificial structure variable will be taken care of by introducing a stopping time until which the limiting solutions, corresponding to the approximations constructed in Section \ref{subsec:approxsol} using ${\bfeta}^n$'s and ${\bfeta}^n_*$'s, are equal. 
\end{remark}
	We are now ready to prove the existence of solutions to the two sub-problems. For this purpose we introduce the following discrete energy and dissipation for $i=0,1$:
	\begin{equation}\label{EnDn}
		\begin{split}
			E^{n+\frac{i}2}&:=\frac12\Big(\int_{\sO}J_*^n|\bu^{n+\frac{i}2}|^2 d\bx
			+\|\bv^{n+\frac{i}2}\|^2_{\bL^2(0,L)}+\|{\bfeta}^{n+\frac{i}2}\|^2_{\bH_0^2(0,L)}\Big),\\
			D_1^{n}&:=\Delta t\int_0^L\ep|\partial_z\bv^{n+\frac12}|^2 dz,\\
			D^n_2&:=\Delta t \int_{\sO}\left( 2\nu J_*^n |\bD^{{\bfeta}_*^{n}}(\bu^{n+1})|^2+\frac1{\ep}|\text{div}^{{\bfeta}^n_*}\bu^{n+1}|^2 \right) d\bx .
	\end{split}\end{equation}
	\begin{lem}({\bf{Existence for the structure sub-problem.}})
		Assume that ${\bfeta}^n$ and $\bv^{n}$ are $\bH^2_0(0,L)$ and $\bL^2(0,L)$ valued $\sF_{t^n}$-measurable random variables, respectively. Then there exist $\bH^2_0(0,L)$- valued $\sF_{t^n}$-measurable random variables ${\bfeta}^{n+\frac12},\bv^{n+\frac12}$ that solve \eqref{first}, and the following semidiscrete energy inequality holds:
		\begin{equation}\label{energy1}
			\begin{split}
				E^{n+\frac12} +D^n_1+C_1^n \leq E^{n},
			\end{split}
		\end{equation}
		where
		$$C^{n}_1:= \frac12\|\bv^{n+\frac12}-\bv^n\|_{\bL^2(0,L)}^2   +\frac12 \|{\bfeta}^{n+\frac12}-{\bfeta}^{n}\|_{\bH_0^2(0,L)}^2,$$ corresponds to numerical dissipation.
	\end{lem}
	\begin{proof}
		The proof of existence and uniqueness of measurable solutions is straightforward (see \cite{KC23}). 
		Furthermore we can write
		$$\bv^{n+\frac12}=\frac{{\bfeta}^{n+\frac12}-{\bfeta}^n}{\Delta t}.$$
		Using this pathwise equality while taking $\bfpsi=\bv^{n+\frac12}$ in \eqref{first}$_3$ and using $a(a-b)=\frac12(|a|^2-|b|^2+|a-b|^2)$,  we obtain,
		\begin{align}
			\notag	\|\bv^{n+\frac12}\|_{\bL^2(0,L)}^2 &+\|\bv^{n+\frac12}-\bv^n\|_{\bL^2(0,L)}^2  +\|{\bfeta}^{n+\frac12}\|_{\bH_0^2(0,L)}^2+\|{\bfeta}^{n+\frac12}-{\bfeta}^n\|_{\bH_0^2(0,L)}^2 +\ep\|\partial_z\bv^{n+\frac12}\|_{\bL^2(0,L)}\\
				&\leq\|\bv^n\|_{\bL^2(0,L)}^2+ \|{\bfeta}^{n}\|_{\bH_0^2(0,L)}^2.\label{energy1_1}
	\end{align}
		Hence using the fact that $\bu^n=\bu^{n+\frac12}$ and adding the relevant terms on both sides of \eqref{energy1_1} we obtain \eqref{energy1}.
	\end{proof}
	\begin{lem}({\bf{Existence for the fluid sub-problem.}})\label{existu}
		For given $\delta>0$, and given $\sF_{t^n}$-measurable random variables $(\bu^{n+\frac12},\bv^n)$ taking values in $\sU$ and $\bv^{n+\frac12}$ taking values in $\bH^2_0(0,L)$, there exists an $\sF_{t^{n+1}}$-measurable random variable $(\bu^{n+1},\bv^{n+1})$ taking values in $\sU$ that solves \eqref{second}, and the solution satisfies the following energy estimate
		\begin{equation}\label{energy2}
			\begin{split}
				E^{n+1}+D_2^{n}+C_2^{n}&\leq E^{n+\frac12} +C\Delta t((P^n_{in})^2+(P^n_{out})^2) 
				+ {C}\|\Delta_nW\|_{U_0}^2
				\|G(\bU^{n},{\bfeta}_*^n)\|_{L_2(U_0;\bL^2)}^2\\
				&+\|
				(G( \bU^{n},{\bfeta}_*^n)\Delta_n W, \bU^{n}) \|+\frac14\int_0^L|\bv^{n+\frac12}-\bv^n|^2dz
		\end{split}\end{equation}
		where
		$$C_2^{n}:=\frac14\int_\sO J_*^n\left( |\bu^{n+1}-\bu^{n}|^2 \right) d\bx +\frac14\int_0^L|\bv^{n+1}-\bv^{n+\frac12}|^2 dz$$
		is numerical dissipation, and ${\bfeta}^n_*$ is as defined in \eqref{eta*}.
	\end{lem}
	\begin{proof}
		The proof of existence and measurability of the solutions is given using Brouwer's fixed point theorem and the Kuratowski and Ryll-Nardzewski selection theorem in \cite{TC23}.

		Next we will show that the solution satisfies energy estimate \eqref{energy1_1}. For this purpose we will derive a pathwise inequality involving the discrete energies. We take $(\bq,\bfpsi)=(\bu^{n+1},\bv^{n+1})$ in \eqref{second} and using the identity $a(a-b)=\frac12(|a|^2-|b|^2+|a-b|^2)$, we obtain
		\begin{align*}
			&\frac12\int_\sO   J_*^n \left(|\bu^{n+1}|^2-|\bu^{n+\frac12}|^2 + |\bu^{n+1}-\bu^{n+\frac12}|^2 \right) 
			+\frac{1}2\int_\sO \left( J_*^{n+1}- J_*^n\right)  |\bu^{n+1}|^2 d\bx\\
			&+2\nu(\Delta t)\int_{\sO_{}} J_*^{n}|\bD^{{\bfeta}_*^{n}}(\bu^{n+1})|^2 d\bx + \frac1{\ep}{(\Delta t)}\int_{\sO_{}} |\text{div}^{{\bfeta}^n_*}\bu^{n+1}|^2 d\bx\\
				&+\frac12\int_0^L|\bv^{n+1}|^2-|\bv^{n+\frac12}|^2+|\bv^{n+1}-\bv^{n+\frac12}|^2 dz \\
			&= (\Delta t)\left( P^n_{{in}}\int_{0}^1u^{n+1}_z\Big|_{z=0}dr-P^n_{{out}}\int_{0}^1u^{n+1}_z\Big|_{z=1}dr\right) \\
			&+( G( \bU^{n},{\bfeta}_*^n)\Delta_n W, (\bU^{n+1}-\bU^{n} ))+
			(G(\bU^{n},{\bfeta}_*^n)\Delta_n W, \bU^{n}).
		\end{align*}
		Also, observe that the discrete stochastic integral is divided into two terms. {\s We estimate the first term by using the Cauchy-Schwarz inequality to obtain that for some} $C(\delta)>0$ independent of $n$ {\s the following holds:} 
		\begin{align*}&| 
			(G(\bU^{n},{\bfeta}_*^n)\Delta_n W, (\bU^{n+1}-\bU^{n} ))| \\
			&\leq {C}\|\Delta_nW\|_{U_0}^2 \|G(\bU^{n},{\bfeta}_*^n)\|_{L_2(U_0,\bL^2)}^2  +\frac14\int_{\sO} J_*^n
		| \bu^{n+1}-\bu^{n}|^2d\bx+\frac18\int_0^L|\bv^{n+1}-\bv^{n}|^2dz\\
			&\leq {C}\|\Delta_nW\|_{U_0}^2 \|G(\bU^{n},{\bfeta}_*^n)\|_{L_2(U_0;\bL^2)}^2 +\frac14\int_{\sO} J_*^n
			| \bu^{n+1}-\bu^{n}|^2d\bx\\
			&+ \frac14\int_0^L|\bv^{n+1}-\bv^{n+\frac12}|^2dz+\frac14\int_0^L|\bv^{n+\frac12}-\bv^{n}|^2dz.
		\end{align*}
 {\s This completes the proof of Lemma \ref{existu}.}
	\end{proof}
	Next, we will obtain uniform estimates on the expectation of the kinetic and elastic energy and dissipation of the coupled problem.
	\begin{thm}{\bf (Uniform Estimates.)}\label{energythm}
		For any $N>0$, $\Delta t=\frac{T}{N}$, and $\delta=(\delta_1,\delta_2)$ satisfying \eqref{delta}, there exists a constant $C>0$ that depends on the initial data, $\delta$, $T$, and is independent of $N$ and $\ep$ such that
		\begin{enumerate}
			\item $\bE\left( \max_{1\leq n \leq N}E^{n}\right)  <C$, $ \bE\left( \max_{0\leq n\leq N-1}E^{n+\frac12}\right) <C.$
			\item $\bE\sum_{n=0}^{N-1} D^n <C$.
			\item $\bE\sum_{n=0}^{N-1}\int_\sO\left(({J}_*^n) |\bu^{n+1}-\bu^{n}|^2 \right) d\bx+\|\bv^{n+1}-\bv^{n+\frac12}\|_{\bL^2(0,L)}^2 <C.$
			\item $\bE \sum_{n=0}^{N-1} \|\bv^{n+\frac12}-\bv^n\|_{\bL^2(0,L)}^2   + \|{\bfeta}^{n+1}-{\bfeta}^{n}\|_{\bH^2(0,L)}^2  <C,$
		\end{enumerate}
		where $E^n$ and $D^n:=D_1^n+D^n_2$ are defined in \eqref{EnDn}.
	\end{thm}
	\begin{proof} We add the energy estimates for the two subproblems \eqref{energy1} and \eqref{energy2} to obtain:
		\begin{equation}
			\begin{split}\label{discreteenergy}
				E^{n+1}+D^{n} &+C^n_1+C^{n}_2\leq E^{n}+ C\Delta t((P^n_{in})^2+(P^n_{out})^2) \\
				&+C\|\Delta_nW\|_{U_0}^2
				\|G(\bU^{n},{\bfeta}_*^n)\|_{L_2(U_0;\bL^2)}^2+\left|
				(G(\bU^{n},{\bfeta}_*^n)\Delta_n W, \bU^{n}) \right|.
			\end{split}
		\end{equation}
		Then for any $m\geq 1$, summing $0\leq n\leq m-1$ gives us
		\begin{equation}\label{energysum}
			\begin{split}
				&E^m+\sum_{n=0}^{m-1}D^{n} +\sum_{n=0}^{m-1}C_1^{n}+\sum_{n=0}^{m-1}C_2^{n}
				\leq E^0+C\,\Delta t\sum_{n=0}^{m-1}\left( (P^n_{{in}})^2 +(P^n_{{out}})^2\right) \\
				&+\sum_{n=0}^{m-1}\left|
				(G(\bU^n,{\bfeta}_*^n)\Delta_n W, \bU^{n} )\right|
				+\sum_{n=0}^{m-1} \|G(\bU^{n},{\bfeta}_*^n)\|_{L_2(U_0;\bL^2)}^2\|\Delta_nW\|_{U_0}^2.
		\end{split}\end{equation}
		
		Next we take supremum over $1 \leq m \leq N$ and then take expectation on both sides of \eqref{energysum}. We begin by treating the right hand side terms. 
	We apply the discrete Burkholder-Davis-Gundy inequality (see e.g. Theorem 2.7 \cite{OPW22}) 
		and using \eqref{boundJ} and  $\bu^{n+\frac12}=\bu^n$, we obtain for some $C(\delta)>0$ that,
		\begin{align*}
			\bE&\max_{1\leq m\leq N}|\sum_{n=0}^{m-1}
			(G(\bU^{n},{\bfeta}_*^n)\Delta_n W, \bU^{n} )| \\
			&\leq {C}\bE\left[\Delta t\sum_{n=0}^{N-1}
			\|G(\bU^{n},{\bfeta}_*^n)\|^2_{L_2(U_0,\bL^2)}\left( \left\|(\sqrt{J_*^n})\bu^{n}\right\|^2_{\bL^2(\sO)}+\|\bv^n\|_{L^2(0,L)}^2\right) \right]^{\frac12}\\
			& \leq {C}{}\bE\left[ \left( \max_{0\leq m\leq N}\left\|(\sqrt{J_*^m})\bu^{m}\right\|^2_{\bL^2(\sO)}+\|\bv^m\|^2_{\bL^2(0,L)}\right) \sum_{n=0}^{N-1}
			\Delta t \left( \|\sqrt{J^n_*}\bu^{n}\|^2_{\bL^2(\sO)}+\|\bv^n\|_{\bL^2(0,L)}^2\right) \right]^{\frac12}\\
			&\leq \frac{1}2\|(\sqrt{J_0})\bu_0\|^2_{\bL^2(\sO_{})} +\frac12\|\bv_0\|^2_{\bL^2(0,L)}+ \frac12\bE\max_{1\leq m\leq N}\left[ \left\|(\sqrt{J_*^m})\bu^{m}\right\|^2_{\bL^2(\sO)}+\|\bv^m\|^2_{\bL^2(0,L)}\right] \\
			&+ {C \Delta t}\bE\left( \sum_{n=0}^{N-1}
			\|(\sqrt{J^n_*})\bu^n\|^2_{\bL^2(\sO)}+\|\bv^n\|_{\bL^2(0,L)}^2\right) ,
		\end{align*}

Using the tower property and \eqref{growthG} for each $n=0,...,m-1$ we obtain
		we write,
		\begin{align}
			\bE[\|G(\bU^n,{\bfeta}_*^n)\|_{L_2(U_0,\bL^2)}^2\|\Delta_nW\|_{U_0}^2]&=\bE[\bE[\|G(\bU^n,{\bfeta}_*^n)\|_{L_2(U_0,\bL^2)}^2\|\Delta_nW\|_{U_0}^2|\sF_n]]\notag\\
			&=\bE[\|G(\bU^n,{\bfeta}_*^n)\|_{L_2(U_0,\bL^2)}^2\bE[\|\Delta_nW\|_{U_0}^2\|\sF_n]]\notag\\
			&=\Delta t(\Tr \bQ)\bE\|G(\bU^n,{\bfeta}_*^n)\|^2_{L_2(U_0,\bL^2)}\notag\\
			&\leq C\Delta t(\Tr \bQ)\bE[\|\sqrt{J_*^n}\bu^n\|^2_{\bL^2(\sO)}+\|\bv^n\|^2_{\bL^2(0,L)}]\label{tower}.
		\end{align}
		Thus we obtain for some $C>0$ depending on $\delta$ and on $\Tr{\bQ}$, that the following holds:
		\begin{equation}
			\begin{split}\label{gronwall}
				\bE\max_{1\leq n\leq N}E^{n}&+ \bE\sum_{n=0}^{N-1}D^n +\bE\sum_{n=0}^{N-1}C_1^{n}+\bE\sum_{n=0}^{N-1}C_2^{n}\leq CE^0+C\|P_{in/out}\|^2_{L^2(0,T)}
				\\&+ C\Delta t \bE\left[\sum_{n=0}^{N-1} \|(\sqrt{J^n_*})\bu^n\|^2_{\bL^2(\sO)}+\|\bv^n\|^2_{\bL^2(0,L)}\right]\\
				&+ \frac12\bE\max_{1\leq n\leq N}\left[ \left\|(\sqrt{J_*^n})\bu^{n}\right\|^2_{\bL^2(\sO)}+\|\bv^n\|^2_{\bL^2(0,L)}\right].
			\end{split}
		\end{equation}
		Hence by absorbing the last term on the right of \eqref{gronwall} we obtain
		\begin{align*}
			\bE\max_{1\leq n\leq N}&\left( \|(\sqrt{J^n_*})\bu^n\|^2_{\bL^2(\sO)} +\|\bv^n\|^2_{\bL^2(0,L)}\right) \leq CE^0+ C\|P_{in/out}\|^2_{L^2(0,T)}\\
			& +\sum_{n=1}^{N-1} \Delta t\bE \max_{1\leq m\leq n}\left(\|(\sqrt{J^m_*})\bu^m\|^2_{\bL^2(\sO)}+\|\bv^m\|^2_{\bL^2(0,L)}\right) .
		\end{align*}
		Applying the discrete Gronwall inequality to $\bE\max_{1\leq n\leq N}(\|(\sqrt{J^n_*})\bu^n\|_{\bL^2(\sO)}^2+\|\bv^n\|^2_{\bL^2(0,L)})$ we obtain,
		$$\bE\max_{1\leq n \leq N}\left( \|(\sqrt{J^n_*})\bu^n\|_{\bL^2(\sO)}^2+\|\bv^n\|^2_{\bL^2(0,L)}\right) \leq C e^T,$$
		where $C$ depends only on the given data and in particular $\delta$.
		
		Hence for $E^n,D^n$ defined in \eqref{EnDn} we have,
		\begin{equation}
			\begin{split}
				\bE\max_{1\leq n\leq N}E^{n}+ \bE\sum_{n=0}^{N-1}D^n&+\bE\sum_{n=0}^{N-1}C_1^{n}+\bE\sum_{n=0}^{N-1}C_2^{n} \\
			&	\leq C(E^0 + \|P_{in}\|^2_{L^2(0,T)}+\|P_{out}\|^2_{L^2(0,T)}
				+ Te^T).
			\end{split}
		\end{equation}
	\end{proof}
	
	\subsection{Approximate solutions}\label{subsec:approxsol}
	In this subsection, we use the solutions $(\bu^{n+\frac{i}2},{\bfeta}^{n+\frac{i}2},\bv^{n+\frac{i}2})$, $i=0,1$, defined for every $N \in \mathbb{N}\setminus \{0\}$ at discrete times to define approximate solutions on the entire interval $(0,T)$. We start by introducing approximate solutions that are piece-wise constant on each sub-interval $ [n\Delta t, (n+1)\Delta t)$ as 
	\begin{align}\label{AppSol}
		\bu_{N}(t,\cdot)=\bu^n, \,
		{\bfeta}_{N}(t,\cdot)={\bfeta}^n, {\bfeta}_{N}^*(t,\cdot)={\bfeta}_*^{n},  \bv_{N}(t,\cdot)=\bv^n,  {\bv}^{\#}_{N}(t,\cdot)=\bv^{n+\frac12}.
	\end{align}
	Observe that all of the processes defined above are adapted to the given filtration $(\sF_t)_{t \geq 0}$.\\
	The following are time-shifted versions of the functions defined in \eqref{AppSol}: 
	\begin{align*}
		\bu^+_{N}(t,\cdot)=\bu^{n+1},\quad {\bfeta}^+_{N}(t,\cdot)={\bfeta}^{n+1}\quad  t \in (n\Delta t, (n+1)\Delta t].
	\end{align*}
	We also define the corresponding piece-wise linear interpolations as follows: for $t \in [t^n,t^{n+1}]$
	\begin{equation}
		\begin{split}\label{approxlinear}
			&\tilde\bu_{N}( t,\cdot)=\frac{t-t^n}{\Delta t} \bu^{n+1}+ \frac{t^{n+1}-t}{\Delta t} \bu^{n}, \quad \tilde \bv_{N}(t,\cdot)=\frac{t-t^n}{\Delta t} \bv^{n+1}+ \frac{t^{n+1}-t}{\Delta t} \bv^{n}\\
			& \tilde{\bfeta}_{N}( t,\cdot)=\frac{t-t^n}{\Delta t} {\bfeta}^{n+1}+ \frac{t^{n+1}-t}{\Delta t} {\bfeta}^{n}, \quad \tilde{\bfeta}^*_{N}( t,\cdot)=\frac{t-t^n}{\Delta t} {\bfeta}^{n+1}_*+ \frac{t^{n+1}-t}{\Delta t} {\bfeta}^{n}_*.
		\end{split}
	\end{equation}
	Observe that,
	\begin{align}\label{etaderi}
		\frac{\partial\tilde{\bfeta}_{N}}{\partial t}=\bv^{\#}_{N},\quad \frac{\partial\tilde{\bfeta}^*_{N}}{\partial t}
		={\sum_{n=0}^{N-1}\theta_\delta({\bfeta}^{n+1})}\bv^{\#}_{N}\chi_{(t^n,t^{n+1})}:=\bv_{N}^*
		\quad a.e. \text{ on } (0,T),  
	\end{align}{\s
		where $\bv^{\#}_N$ was introduced in \eqref{AppSol}.}
	\begin{lem}\label{bounds}
		Given
		$\bu_0 \in \bL^2(\sO_{})$, ${\bfeta}_0 \in  \bH^2_0(0,L)$, $\bv_0 \in \bL^2(0,L)$, 
		for a fixed $\delta$, we have that
		\begin{enumerate}
			\item $\{{\bfeta}_{N}\},\{{\bfeta}_{N}^*\}$ and thus $\{\tilde {\bfeta}_{N}\},\{\tilde{\bfeta}_{N}^*\}$ are bounded independently of $N$ and $\ep$ in\\ $L^2(\Omega;L^\infty(0,T;\bH^2_0(0,L)))$.
			\item $\{\bv_{N}\},\{\bv_{N}^{\#}\},\{\bv_{N}^{*}\}$ are bounded independently of $N$ and $\ep$ in $L^2(\Omega;L^\infty(0,T;\bL^2(0,L)))$.
			\item $\{\bu_{N}\}$ is bounded independently of $N$ and $\ep$ in $L^2(\Omega;L^\infty(0,T;\bL^2(\sO)))$.
			\item $\{\bu^+_{N}\}$ is bounded independently of $N$ and $\ep$ in 
			$		 L^2(\Omega;L^2(0,T;V) \cap L^\infty(0,T;\bL^2(\sO))).$
			\item $\{\frac1{\sqrt{\ep}}\text{div}^{{\bfeta}^*_{N}}\bu^+_{N}\}$ is bounded independently of $N$ and $\ep$ in $L^2(\Omega;L^2(0,T;L^{2}(\sO)))$.
					\item $\{\bv^+_{N}\}$ is bounded independently of $N$ in 
				$		 L^2(\Omega;L^2(0,T;\bH^{\frac12}(0,L)) ).$
			\item $\{\sqrt{\ep}\bv^{\#}_{N}\}$ and $\{\sqrt{\ep}\bv^{*}_{N}\}$ are bounded independently of $N$ and $\ep$ in  $L^2(\Omega;L^2(0,T;{\bH^1_0(0,L)}))$.
				\end{enumerate} 
	\end{lem}
	\begin{proof}
		In order to prove (4) observe that for each $\omega\in \Omega$, $\nabla\bu^{n+1}=\nabla^{{\bfeta}^{n}_*}\bu^{n+1} (\nabla A^\omega_{{\bfeta}^{n}_*})$. Thus we have, 
		\begin{align*}
			\delta_1 \bE\int_{\sO}|\nabla\bu^{n+1}|^2d\bx &\leq  \bE\int_{\sO}(J_*^{n})|\nabla\bu^{n+1}|^2d\bx=\bE\int_{\sO}(J^{n}_*)|\nabla^{{\bfeta}_*^{n}}\bu^{n+1}\cdot \nabla A^\omega_{{\bfeta}_*^{n}}|^2d\bx\\
			&\leq C(\delta)\bE\int_{\sO}(J^{n}_*)|\nabla^{{\bfeta}_*^{n}}\bu^{n+1}|^2d\bx \leq {K}C(\delta)\bE\int_{\sO}(J_*^{n})|\bD^{{\bfeta}_*^{n}}\bu^{n+1}|^2d\bx,
		\end{align*}
		where $K>0$ is the universal Korn constant that depends only on the reference domain $\sO$. This result follows from Lemma 1 in \cite{V12} and because of uniform bounds \eqref{boundJ} which imply that $\{ A^\omega_{\bfeta^*_N}(t); \,\omega\in\Omega, \,t\in [0,T] \}$ is compact in $\bW^{1,\infty}(\sO)$.
		Thus, there exists $C>0$, independent of $N$, such that
		\begin{align}\label{uboundV}
			\bE\int_0^T\int_{\sO}|\nabla\bu^{+}_{N}|^2d\bx ds =\bE\sum_{n=0}^{N-1} \Delta t\int_{\sO}|\nabla\bu^{n+1}|^2d\bx \leq C(\delta).
		\end{align}
			Statement (6) is then a result of Statement (4) and the fact that, by construction, $\bv_N^+$ is the trace of $\bu^+_N$ on $\Gamma$. 
			The proofs of the rest of the statements follow directly from Theorem \ref{energythm}. 
\end{proof}
	\section{Passing $N\rightarrow\infty$}\label{sec:limit}
	In this section we will obtain almost sure convergence of the stochastic approximate solutions, which is required to be able to pass $N \rightarrow \infty$. We use compactness arguments and establish tightness of the laws of the approximate random variables defined in Section \ref{subsec:approxsol}.
\subsection{{ Tightness results}} \label{sec:tight}
Given our stochastic setting we do not expect the fluid and structure velocities to be differentiable.
The tightness results, i.e. Lemmas \ref{tightuv} and \ref{tightl2} below, will rely on an application of the Aubin-Lions theorem and the following variant \cite{Tem95}:
\begin{lem}\label{compactLp}
	Let	 the translation in time by $h$ of a function $f$ be denoted by:
	$$T_h f(t,\cdot)=f(t-h,\cdot), \quad h\in \R.$$ Assume that $\mathcal{Y}_0\subset\mathcal{Y}\subset\mathcal{Y}_1$ are Banach spaces such that $\mathcal{Y}_0$ and $\mathcal{Y}_1$ are reflexive with compact embedding of $\mathcal{Y}_0$ in $\mathcal{Y}$, 
	then for any $m>0$, the embedding	$$ \{\bu:\bu \in L^2(0,T;\mathcal{Y}_0):\sup_{0<h< T} \frac1{h^m}\|T_h\bu-\bu\|_{L^2(h,T;\sY_1)}<\infty\} \hookrightarrow L^2(0,T;\mathcal{Y}),$$	is compact.
\end{lem}

\begin{lem}\label{tightuv} The laws of $\bu^+_N$ and $\bv^+_N$ are tight in $L^2(0,T;\bH^\alpha(\sO))$, and $ L^2(0,T;\bH^{\beta}(0,L))$, respectively for any $0\leq \alpha<1$ and $0\leq \beta<\frac12$. 
\end{lem}
\begin{proof} 
	The aim of this proof is to apply Lemma \ref{compactLp} 
	by obtaining appropriate bounds for
	$$\int_h^{T} \|T_h\bu_N-\bu_N\|^2_{\bL^2(\sO)}+\|T_h\bv_N-\bv_N\|^2_{\bL^2(0,L)}=	(\Delta t)\sum_{n=j}^N \|\bu^n-\bu^{n-j}\|^2_{\bL^2(\sO)} + \|\bv^n-\bv^{n-j}\|^2_{\bL^2(0,L)},$$
	for any $N$ and $h$. Here $1\leq j \leq N$ such that $h=j\Delta t-s$ for some  $s<\Delta t$.
	
	To achieve this goal, we will construct appropriate test functions for equations \eqref{first} and \eqref{second} that will result into the term on the right hand side of the equation above. This has to be done carefully since the fluid variables are all defined on different domains and because we can not test equation \eqref{first}$_3$ with $\bv^k$, as it does not have the required regularity. Thus, we will use space mollification to arrive at the desired test function. For the mollified functions to still satisfy the required 0 boundary conditions on $\partial\sO\setminus \Gamma$ we will apply an appropriate horizontal squeezing operator (see \cite{CDEG}, \cite{MC19}, \cite{TC23}).
	
	Hence our plan is as follows: assume that $\sO_\delta=(0,L)\times (0,R(\delta))$ is the maximal rectangular domain  consisting of all the fluid domains associated with the structure displacements $\bfeta^*_N$ for any $N$ and $t,\omega$. We fix an $N$ and $\omega\in\Omega$, and for any $0\leq k,n \leq N$ we pull back $\bu^k$ on the physical domain $\sO_{\bfeta^{k-1}_*}$, extend it, then "squeeze" it horizontally in a way that its divergence is preserved, mollify it in a way that 0 boundary conditions are preserved, and then push it to the domain $\sO$ via the ALE map $A^\omega_{\bfeta^n_*}$. This way we can transform $(\bu^k,\bv^k)$ in a way that it can be used as a test function for the equations \eqref{first} and \eqref{second} for $(\bu^n,\bv^n)$.
	
	First, we let $ \tilde\bu_{}^k=\bu^k\circ A^{-1}_{\bfeta^{k-1}_*}$. Observe that there exists $\bar\bu^k$, a divergence free extension of $\tilde\bu^k$ such that $\bar\bu^k=0$ when $z=0,L$, $\bar\bu^k|_{\Gamma_{\bfeta^{k-1}_*}}=\bv^k$ and $\bar\bu^k=(0,\bar v^k)$ when $r=R(\delta)$ where $\bar v^k=\bv^k\times \partial_z\bfeta^{k-1}_*$ such that $\|\bar\bu^k\|_{\bH^1(\sO_\delta)} \leq C\|\tilde\bu^k\|_{\bH^1(\sO_{\bfeta^{k-1}_*})}$ where $C>0$ depends only on $\delta$.	
	These boundary conditions ensure that $\int_{\partial(\sO_\delta\setminus \sO_{\bfeta^{k-1}_*})}\bar\bu^k\cdot\bn^{\bfeta^{k-1}_*}=0$.
	
	 Next we denote by $\bu^{k,ext}$ its  extension by $\bar\bu^k$ in $\sO_\delta\setminus\sO_{\bfeta^{k-1}_*}$ and 0 everywhere else.
	
	Then for 
	$1<\sigma<2$, we define the squeezed version as,
	\begin{align*}
		\bu^{k}_\sigma(z,r)&:= (u^{k,ext}_z,\sigma u^{k,ext}_r)\circ \sigma (z,r) , \quad \text{ where }
		\sigma(z,r)= \sigma(z-\frac{L}2)+\frac{L}2 .
	\end{align*}
		Observe that $\|\bu_{\sigma}^k\|_{\bH^1(\sO_\delta)} \leq C(\delta) \|\bu^{k,ext}\|_{\bH^1(\sO_\delta)} \leq C(\delta) \|\bu^k\|_{\bH^1(\sO)}$.

 Note here that we scale the $z$--coordinate of the function i.e. squeeze it horizontally, so that mollification in the next step does not ruin its boundary conditions at $\Gamma_{l/r}$. Observe that we need to "squeeze" the function $\bu^{k,ext}$ around the bottom boundary $\Gamma_b$ as well and it can be done using the same argument. However we choose leave it out of our discussion for a cleaner presentation.
	
	As we look for an appropriate test function, we introduce a space regularization  $\bu^k_{\sigma,\lambda}$, using standard 2D mollifiers. Finally,	 for any $\lambda<\frac{L(\sigma-1)}{2\sigma}$  we define 
	\begin{align}
		\bu^{k,n}_{\sigma,\lambda}=\bu^k_{\sigma,\lambda}\circ A^\omega_{\bfeta^n_*}.
	\end{align}
	Hence we have that $\bu^{k,n}_{\sigma,\lambda}$, for any $n-j\leq k\leq n$, satisfies the correct boundary conditions on $\partial\sO\setminus\Gamma$ and, \begin{align}\label{kton}\|\text{div}^{\eta^n_*}\bu^{k+1,n}_{\sigma,\lambda}\|_{L^2(\sO)}\leq {C(\delta)}{}\|\text{div}^{\eta^k_*}\bu^{k+1}\|_{L^2(\sO)},
	\end{align}
For the rest of the proof we will fix,
$\sigma-1=h^{\frac16}$ and $\lambda=\frac{L}{4}h^{\frac16}.$
We define $\bv^{k,n}_{\sigma,\lambda}(z)=\bu^{k,n}_{\sigma,\lambda}(z,1)$ for $z\in [0,L]$. Now, for simplicity of notation, let  $ \bfvarphi_n(z)=(z+(\eta^n_*)_z,1+(\eta^n_*)_r)$ and observe that,
\begin{align}	\notag\|\partial_{zz}\bv^{k,n}_{\sigma,\lambda}&\|_{\bL^2(0,L)}=\| \partial_{zz}(\bu^k_{\sigma,\lambda}\circ \bfvarphi_n)\|_{\bL^2(0,L)}\\
	&\notag\leq C(\delta)\|\bD^2\bu^k_{\sigma,\lambda}\circ \bfvarphi_n\|_{\bL^2(0,L)}+\|\nabla\bu^k_{\sigma,\lambda}\circ \bfvarphi_n\partial_{zz} \bfvarphi_n\|_{\bL^2(0,L)}\\
	&\notag\leq C(\delta)\|\bD^2\bu^k_{\sigma,\lambda}\circ \bfvarphi_n\|_{\bL^2(0,L)}+\|\nabla\bu^k_{\sigma,\lambda}\circ \bfvarphi_n\|_{\bH^\alpha(0,L)}\|\partial_{zz} \bfvarphi_n\|_{\bL^2(0,L)}, \,\,\quad \alpha>\frac12\\
	&\leq C(\delta)(\|\bD^2\bu^k_{\sigma,\lambda}\|_{\bH^\alpha(\sO_{\bfeta^n_*})}+\|\nabla\bu^k_{\sigma,\lambda}\|_{\bH^{\alpha+\frac12}(\sO_{\bfeta^n_*})}\|\partial_{zz}\bfeta^n_*\|_{\bL^2(0,L)}),\quad \frac12<\alpha<1.\label{vh2}
\end{align}
In the last estimate, the constant depends only on $\delta$ thanks to the ideas developed in the trace theorem in \cite{G08} applied along with the results of \cite{D96}.
We know that $\|\phi-\phi_\lambda\|_{L^2} \leq \lambda^m\|\phi\|_{H^m}$ and that
$\|T_h\bfeta^*_N-\bfeta^*_N\|_{L^\infty(0,T;\bL^\infty(0,L))}\leq \|\tilde\bfeta^*_N\|_{C^{0,\frac14}(0,T;\bL^\infty(0,L))} h^\frac14$.
Hence, 
for any $n-j\leq k\leq n$, we obtain
\begin{align}
	\notag	\|\bu^{k,n}_{\sigma,\lambda}-\bu^{k}&\|_{\bL^2(\sO)} \leq \|\bu^{k}_{\sigma,\lambda}-\bu^{k}_\sigma\|_{\bL^2(\sO_{\bfeta^n_*})}+\|\bu^k_\sigma-\bu^{k,ext}\|_{\bL^2(\sO_{\bfeta^n_*})}+\|\bu^{k,ext}\circ(A_{\bfeta^n_*}-A_{\bfeta^{k-1}_*})\|_{\bL^2(\sO)}\\
	\notag	& \leq   \lambda\|\bu^{k}_\sigma\|_{\bH^1(\sO_{\bfeta^n_*})} +(\sigma-1)^{\frac12}\|\nabla\bu^{k,ext}\|_{\bL^2(\sO_\delta)}
	+\|\nabla\bu^{k,ext}\|_{\bL^2(\sO_\delta)}\|\bfeta^n_*-\bfeta^{k-1}_*\|_{\bL^\infty(0,L)}\\
	& \leq h^{\frac1{12}} \|\bu^{k}\|_{\bH^1(\sO)}
	+h^{\frac14}\|\tilde\bfeta^*_N\|_{C^{0,\frac14}(0,T;\bL^\infty(0,L))}\|\bu^{k}\|_{\bH^1(\sO)}. \label{u2}
\end{align}
Here, for the second term, we also used  Lemma 4.7 in \cite{MC19} (see also the following computation).
 Let $ \bfvarphi_n(z)=(z+(\eta^n_*)_z,1+(\eta^n_*)_r)$. 
 We then obtain,
\begin{align}
	\notag	\|&\bv^{k,n}_{\sigma,\lambda}-\bu^{k,ext}|_{\Gamma_{\bfeta^{n}_*}}\|^2_{\bL^2(0,L)}=\int_{0}^L\Big|\int_{B(0,1)}\rho(y)(\bu^k_{\sigma}(\bfvarphi_n(z)+\lambda y)-\bu^{k,ext}(\bfvarphi_n(z))) dy\Big|^2 dz\\
	&\notag=\int_{0}^L\Big|\int_{B(0,1)}\rho(y)(\bu^{k,ext}(\sigma(\bfvarphi_n (z))+\lambda \sigma(y)) -\bu^{k,ext}(\bfvarphi_n(z))) dy\Big|^2 dz\\			&\notag=\int_{0}^L\int_{B(0,1)}|(\bu^{k,ext}(\sigma(\bfvarphi_n (z))+\lambda \sigma(y)) -\bu^{k,ext}(\bfvarphi_n(z)))|^2 dy dz\\		&\notag\leq\int_{0}^L\int_{B(0,1)}\Bigg|\int^{\sigma(\bfvarphi_n (z))+\lambda \sigma(y)}_{\bfvarphi_n(z)}|\nabla \bu^{k,ext}(w)|dw\Bigg|^2 dy dz\\	
	&\notag\leq\int_{0}^L\int_{B(0,1)}|{\sigma(\bfvarphi_n (z))+\lambda \sigma(y)}-{\bfvarphi_n(z)}|\int^{\sigma(\bfvarphi_n (z))+\lambda \sigma(y)}_{\bfvarphi_n(z)}|\nabla \bu^{k,ext}(w)|^2dw  dy dz\\	
	&\leq C(\delta,L)((\sigma-1)+\lambda )  \int_{\sO^{}_\delta}|\nabla\bu^{k,ext}|^2 dzdr \leq Ch^{\frac16}\|\bu^k\|^2_{\bH^1(\sO)}.\label{v1}
\end{align}
		
	Then, adapting ideas from \cite{CDEG}, for any $n \leq N$ we "test" \eqref{first}$_3$ and \eqref{second}$_2$ with 
	\begin{align}\label{Qn}
		\bQ_n:=(\bq_n,\bfpsi_n)=\left( (\Delta t) \sum_{k=n-j+1}^n\bu^{k,n}_{\sigma,\lambda},\,\, (\Delta t) \sum_{k=n-j+1}^n\bv_{\sigma,\lambda}^{k,n}\right) .
	\end{align}
	
	This gives us
	\begin{align*}	 	
		&\int_{\sO}\left((J^{n+1}_*) \bu^{n+1}-(J^{n}_*)\bu^{n}\right) \left( \Delta t \sum_{k=n-j+1}^n\bu^{k,n}_{\sigma,\lambda}\right)   
		+\int_0^L(\bv^{n+1}-\bv^{n} )\left( \Delta t \sum_{k=n-j+1}^n\bv_{\sigma,\lambda}^{k,n}\right) \\
		&-\frac{1}2\int_\sO \left( J_*^{n+1}-J_*^n\right)  \bu^{n+1}\cdot \left( \Delta t \sum_{k=n-j+1}^n\bu^{k,n}_{\sigma,\lambda}\right) +(\Delta t) b^{\bfeta_n^*}(\bu^{n+1},\bw_*^n,\left( \Delta t \sum_{k=n-j+1}^n\bu^{k,n}_{\sigma,\lambda}\right) ) \\
		&+ \frac{(\Delta t)}{\ep}\int_\sO 
		\text{div}^{\bfeta^n_*}\bu^{n+1}\text{div}^{\bfeta^n_*}\left( \Delta t \sum_{k=n-j+1}^n\bu^{k,n}_{\sigma,\lambda}\right) +(\Delta t)\left(  \sL_e(\bfeta^{n+\frac12}) , \left( \Delta t \sum_{k=n-j+1}^n\bv_{\sigma,\lambda}^{k,n}\right) \right) \\
		&
		+2\nu(\Delta t)\int_{\sO}J^n_* \bD^{\eta_*^{n}}(\bu^{n+1})\cdot \bD^{\eta_*^{n}}\left( \Delta t \sum_{k=n-j+1}^n\bu^{k,n}_{\sigma,\lambda}\right)  
		 + \ep(\Delta t)\int_0^L\partial_z\bv^{n+\frac12} \partial_z \left( \Delta t \sum_{k=n-j+1}^n\bv^{k,n}_{\sigma,\lambda}\right) \\
		&= (\Delta t)\left( P^n_{{in}}\int_{0}^1q^n_z\Big|_{z=0}dr-P^n_{{out}}\int_{0}^1q^n_z\Big|_{z=1}dr\right) +  
		(G(\bU^{n},\bfeta_*^n)\Delta_n W, \bQ_n).
	\end{align*}
	Observe that summation by parts formula gives us for the first two terms,
	\begin{align*}
		&-\sum_{n=0}^N	\int_{\sO}\left((J^{n+1}_*) \bu^{n+1}-(J^{n}_*)\bu^{n}\right) \left( \Delta t \sum_{k=n-j+1}^n\bu^{k,n}_\sigma\right)   
		+\int_0^L(\bv^{n+1}-\bv^{n} )\left( \Delta t \sum_{k=n-j+1}^n\bv_{\sigma,\lambda}^{k,n}\right) \\
		&=(\Delta t)\sum_{n=1}^N\left( \int_{\sO}(J^{n}_*) \bu^{n}(\bu^{n,n}_{\sigma,\lambda}-\bu_{\sigma,\lambda}^{n-j,n})d\bx+\int_0^L \bv^n(\bv_{\sigma,\lambda}^{n,n}-\bv_{\sigma,\lambda}^{n-j,n})dz\right)\\
		&- \int_\sO (J^{N+1}_*)\bu^{N+1} \left( \Delta t \sum_{k=N-j+1}^N\bu^{k,N}_\sigma\right)-\int_0^L \bv^N \left( \Delta t \sum_{k=N-j+1}^N\bv_{\sigma,\lambda}^{k,n}\right),
		\shortintertext{where the first term on the right side can be written as,}
		&(\Delta t)\sum_{n=0}^N\left(
		\int_{\sO}(J^{n}_*) \bu^{n}(\bu^{n}-\bu^{n-j})d\bx+ \int_{\sO}(J^{n}_*) \bu^{n}(\bu^{n,n}_\sigma-\bu^{n}-(\bu_\sigma^{n-j,n}-\bu^{n-j}))d\bx\right)\\
		&+(\Delta t)\sum_{n=0}^N\int_0^L \bv^n(\bv^n-\bv^{n-j})dz+\int_0^L \bv^n(\bv^{n,n}_{\sigma,\lambda}-\bv^n-(\bv^{n-j,n}_{\sigma,\lambda}-\bv^{n-j}))\\
		&=(\Delta t)\sum_{n=0}^N\left(
		\frac12\int_{\sO}(J^{n}_*) (|\bu^{n}|^2-|\bu^{n-j}|^2+|\bu^{n}-\bu^{n-j}|^2)d\bx\right) \\
		&+(\Delta t)\sum_{n=0}^N \left( \int_{\sO}(J^{n}_*) \bu^{n}(\bu^{n,n}_\sigma-\bu^{n}-(\bu_\sigma^{n-j,n}-\bu^{n-j}))d\bx\right)\\
		&+(\Delta t)\sum_{n=0}^N\int_0^L\frac12 (|\bv^n|^2-|\bv^{n-j}|^2+|\bv^n-\bv^{n-j}|^2)+ \bv^n(\bv^{n,n}_{\sigma,\lambda}-\bv^n-(\bv^{n-j,n}_{\sigma,\lambda}-\bv^{n-j})).
	\end{align*}	 
	where we set $\bu^n=0$ and $\bv^n=0$ for $n<0$ and $n>N$.
	Observe that the right hand side above will give us the desired terms. In what follows we will treat rest of the terms. First,
	\begin{align*}
	I_0:=	\Delta t\sum_{n=0}^N\int_0^L |\bv^n|^2-|\bv^{n-j}|^2 =\Delta t \sum_{n=N-j+1}^{N}\|\bv^n\|^2_{\bL^2(0,L)}\leq h\max_{0\leq n \leq N}\|\bv^n\|^2_{\bL^2(0,L)}.
	\end{align*} 
	Hence,
	\begin{align*}
		\bP(|I_0|>M) \leq \frac{h}M \bE[\max_{0\leq n \leq N}\|\bv^n\|^2_{\bL^2(0,L)}] \leq \frac{Ch}{M}.
	\end{align*}
	Before moving on to the next term, we recall that for two matrices $A$ and $B$, the derivative of the determinant $D(det)(A)B=det(A)tr(BA^{-1})$. Hence applying the mean value theorem to $det(\nabla A_{\bfeta^n_*})-det(\nabla A_{\bfeta^{n+j}_*})$, using \eqref{boundJ} and the fact that $det(A) \leq (\max A_{ij})^2$, we obtain, for some $\beta\in[0,1]$, that (for details see (73) in \cite{MC16})
\begin{equation}
	\begin{split}\label{boundJt}
		|{J^{n+j}_*-J^{n}_*}|&= |\text{det}(\nabla A^{n,\beta})\nabla^{n,\beta}\cdot(A_{\bfeta^{n+j}_*}-A_{\bfeta^{n}_*})|\\
		& \leq C(\delta)\|A_{\bfeta^{n+j}_*}-A_{\bfeta^{n}_*}\|_{\bC^1(\sO)},
	\end{split}
\end{equation}
	where $\nabla^{n,\beta}=\nabla^{\bfeta^n_*} +\beta(\nabla^{\bfeta^{n+j}_*}-\nabla^{\bfeta^n_*})$ and $\nabla A^{n,\beta}=\nabla A_{\bfeta^n_*} +\beta(\nabla A_{\bfeta^{n+j}_*}-\nabla A_{\bfeta^n_*})$.\\
Hence, using \eqref{boundJ} again, we find the following bounds for any $s>\frac32$,
	\begin{align*}
		I_1&:=	(\Delta t)\sum_{n=0}^N\left(
		\int_{\sO}(J^{n}_*) (|\bu^{n}|^2-|\bu^{n-j}|^2)d\bx\right)\\
		&=	(\Delta t)\left( \sum_{n=N-j+1}^N\int_{\sO}(J^{n}_*)|\bu^n|^2 d\bx	+\sum_{n=0}^{N-j}\int_{\sO}(J^{n}_*-J^{n+j}_*) |\bu^{n}|^2d\bx\right) \\
					&\leq	(\Delta t) \sum_{n=N-j+1}^N\int_{\sO}(J^{n}_*)|\bu^n|^2 d\bx	+	(\Delta t)\sum_{n=0}^{N-j}\| A_{\bfeta^{n+j}_*}- A_{\bfeta^n_*}\|_{\bC^1(\sO)} \|\bu^{n}\|^2_{\bL^2(\sO)} \\
		&\leq   \left( h\sup_{1\leq n \leq N}\int_{\sO}(J^{n}_*)|\bu^n|^2 d\bx+ \sup_{0\leq k \leq N-1}\|\bu^{n}\|^2_{\bL^2(\sO)}h^{\frac14}\| \bfeta^{n+j}_*-\bfeta^n_*\|_{\bH^s(0,L)} \right)\\
		&\leq  h^{\frac14}  \sup_{0\leq k \leq N-1}\|\bu^{n}\|^2_{\bL^2(\sO)}\left(1+\|\tilde \bfeta^*_N\|_{C^{0,\frac14}(0,T;\bH^s(0,L))} \right).		
	\end{align*}
	Hereon we will repeatedly use the fact that for any two positive random variables $A$ and $B$, $\{\omega: A+B>M\} \subseteq \{\omega: A>\frac{M}2\}\cup \{\omega: B>\frac{M}2\}$ which implies that
	\begin{align*}
		\bP(|A|+|B|>M) \leq \bP(|A|>\frac{M}2)+\bP(|B|>\frac{M}2),
		\shortintertext{and similarly,}
		\bP(|AB|>M) \leq \bP(|A|>\sqrt{M})+\bP(|B|>\sqrt{M}).
	\end{align*}
	Hence, the embedding $W^{1,\infty}(0,T;L^2(0,L))\cap L^2(0,T;H^2_0(0,L)) \hookrightarrow C^{0,\frac14}(0,T;H^\frac32(0,L))
	$, gives us for any $2>s\geq \frac32$ that
	\begin{align*}
		\bP&\left(|I_1|>M \right) \leq \bP\left( h^{\frac18}\sup_{1\leq n \leq N}\|\bu^n\|^2_{\bL^2(\sO)} \geq \sqrt{M}\right) + \bP\left( h^{\frac18}\|\tilde\bfeta^*_N\|_{C^{0,\frac14}(0,T;\bH^s(0,L))}\geq \sqrt{M} \right)\\
		&\leq \frac{{h}^{\frac18}}{\sqrt{M}}\bE\left( \sup_{1\leq n \leq N}\|\bu^n\|^2_{\bL^2(\sO)}  +\|\tilde\bfeta^*_N\|_{C^{0,\frac14}(0,T;\bH^s(0,L))} \right)\\
		&\leq \frac{{h}^{\frac18}}{\sqrt{M}}\bE\left( \sup_{1\leq n \leq N}\|\bu^n\|^2_{\bL^2(\sO)}  +\|\tilde\bfeta^*_N\|_{W^{1,\infty}(0,T;\bL^2(0,L))\cap L^2(0,T;\bH^2_0(0,L)) } \right) \leq C\frac{{h}^{\frac18}}{\sqrt{M}}.
	\end{align*}

	Next, thanks to \eqref{u2}, we see that
	\begin{align*}
		& I_2=	|(\Delta t)\sum_{n=0}^N \left( \int_{\sO}(J^{n}_*) \bu^{n}(\bu^{n,n}_{\sigma,\lambda}-\bu^{n}-(\bu_{\sigma,\lambda}^{n-j,n}-\bu^{n-j}))d\bx\right)|\\
		& \leq C(\delta)(\Delta t)\sum_{n=0}^N\|\bu^n\|_{\bL^2(\sO)}(\|\bu^{n,n}_{\sigma,\lambda}-\bu^n\|_{\bL^2(\sO)}+\|\bu^{n-j,n}_{\sigma,\lambda}-\bu^{n-j}\|_{\bL^2(\sO)})\\
		&\leq C(\delta, T) h^\frac1{12} \left(1
		+\|\tilde\bfeta^*_N\|_{C^{0,\frac14}(0,T;\bL^\infty(0,L))}\right) \,\sup_{1\leq n\leq N} \|\bu^n\|_{\bL^2(\sO)} \left( (\Delta t)\sum_{n=1}^{N} \|\bu^{k}\|^2_{\bH^1(\sO)}\right)^{\frac12} .
	\end{align*}
	Hence, we find that
	\begin{align*}
		\bP(|I_2|>M)& \leq
		\frac{h^{\frac1{18}}}{M^{\frac23}}\bE\left( \|\tilde\bfeta^*_N\|^2_{C^{0,\frac14}(0,T;\bL^\infty(0,L))}+ \sup_{1\leq n\leq N}\|\bu^n\|^2_{\bL^2(\sO)}+ \left((\Delta t) \sum_{n=0}^N 
		\|\bu^{n}\|^2_{\bH^1(\sO)}\right)\right) \\
		&\leq C\frac{h^{\frac1{18}}}{M^{\frac23}}.
	\end{align*}
Recall that $\bv^{k,n}_{\sigma,\lambda}$ and $\bv^k$ are the traces of $\bu^{k}_{\sigma,\lambda}$ on $\Gamma_{\bfeta^n_*}$ and $\tilde\bu^k$ (and thus $\bu^{k,ext}$) on $\Gamma_{\bfeta^{k-1}_*}$ respectively.  Hence using \eqref{v1} we obtain
	\begin{align*}
		\|(\bv^{n-j,n}_{\sigma,\lambda}-\bv^{n-j})\|_{\bL^2(0,L)} &\leq 	\|\bv^{n-j,n}_{\sigma,\lambda}-\bu^{n-j,ext}|_{\Gamma_{\bfeta^n_*}}\|+\|\bu^{n-j,ext}|_{\Gamma_{\bfeta^n_*}}-\bu^{n-j,ext}|_{\Gamma_{\bfeta^{n-j-1}_*}}\|\\
		& \leq C(\delta) h^{\frac18} \|\bu^k\|_{\bH^1(\sO)} + \|\bfeta^n_*-\bfeta^{n-j}_*\|_{\bL^\infty(0,L)} \|\bu^k\|_{\bH^1(\sO)}\\
		&\leq C(h^{\frac1{12}}+h^{\frac14}\|\tilde\bfeta^*_N\|_{C^{0,\frac14}(0,T;\bH^s(0,L))})\|\bu^k\|_{\bH^1(\sO)}.
	\end{align*} 
Therefore,
	\begin{align*}
		I_3&:=	|(\Delta t)\sum_{n=0}^N \int_0^L \bv^n(\bv^{n,n}_{\sigma,\lambda}-\bv^n-(\bv^{n-j,n}_{\sigma,\lambda}-\bv^{n-j}))| \\
		&\leq 	(\Delta t)\sup_{0\leq n\leq N} \|\bv^n\|_{\bL^2(0,L)}\sum_{n=0}^N  \|\bv^{n,n}_{\sigma,\lambda}-\bv^n\|+\|(\bv^{n-j,n}_{\sigma,\lambda}-\bv^{n-j})\| \\
		& \leq	C\left( h^{\frac1{12}}+h^{\frac14}\|\tilde\bfeta^*_N\|_{C^{0,\frac14}(0,T;\bH^s(0,L))}\right)   \sup_{0\leq n\leq N}\|\bv^n\|_{\bL^2(0,L)}\left( (\Delta t)\sum_{n=0}^N \|\bu^k\|_{\bH^1(\sO)}^2\right) ^{\frac12}.
	\end{align*}
	Hence we obtain,
	\begin{align*}
		\bP(|I_3|>M) &\leq 	\frac{h^{\frac1{18}}}{M^{\frac23}}\bE\left( \|\tilde\bfeta^*_N\|^2_{C^{0,\frac14}(0,T;\bH^s(0,L))}  + \sup_{0\leq n\leq N}\|\bv^n\|^2_{\bL^2(0,L)} + (\Delta t)\sum_{n=0}^N \|\bv^n\|^2_{\bH^\frac12(0,L)}\right) \\
		&\leq C\frac{h^{\frac1{18}}}{M^{\frac23}}.
	\end{align*}
	For the penalty term, thanks to \eqref{kton}, we observe that,
	\begin{align*}
		& I_4:=|	\frac{(\Delta t)}{\ep}\sum_{n=0}^{N}\int_\sO 
		\text{div}^{\bfeta^n_*}\bu^{n+1}\left( \Delta t \sum_{k=n-j+1}^n\text{div}^{\bfeta^n_*}(\bu^{k,n}_{\sigma,\lambda})\right) d\bx |  \\
		&\leq \frac{(\Delta t)}{\ep}(\Delta t)\sum_{n=0}^{N}
		\|	\text{div}^{\eta^n_*}\bu^{n+1}\|_{L^2(\sO)}\left( 	{(\Delta t)} \sum_{k=n-j+1}^n\|\text{div}^{\bfeta^n_*}(\bu^{k,n}_{\sigma,\lambda})\|^2_{L^2(\sO)}\right)^{\frac12} \sqrt{h} \\
		&\leq C \sqrt{hT}  \left(  	\frac{(\Delta t)}{\ep} \sum_{n=0}^N\|\text{div}^{\bfeta^{n}_*}(\bu^{n+1})\|^2_{L^2(\sO)}\right) .
	\end{align*}
	Hence,
	\begin{align*}
		\bP(I_4>M) &\leq \frac{\sqrt{h}}{M}
	\bE\left(  	\frac{(\Delta t)}{\ep} \sum_{n=0}^N\|\text{div}^{\bfeta^{n}_*}(\bu^{n+1})\|^2_{L^2(\sO)}\right)\leq C\frac{h^{\frac12}}{M^{}}.
	\end{align*}
	Notice that, due to Theorem \ref{energythm} (2), the constant $C$ in the estimate above does not depend on $\ep$. 
		Similarly since $\|\phi_\lambda\|_{H^1} \leq C\|\phi\|_{H^1}$, the term with the transformed symmetrized gradients yields a similar result.
	\begin{align*}
	\bE[I_5] =&\bE|	{(\Delta t)}{}\sum_{n=0}^{N}\int_\sO 
		\bD^{\eta^n_*}\bu^{n+1}\left( \Delta t \sum_{k=n-j+1}^n\bD^{\eta^n_*}(\bu^{k,n}_{\sigma,\lambda})\right) d\bx |  \\
		&\leq \bE((\Delta t)\sum_{n=0}^{N}
		\|	\bD^{\eta^n_*}\bu^{n+1}\|_{\bL^2(\sO)}\left( {(\Delta t)} \sum_{k=n-j+1}^n\|\bD^{\eta^n_*}(\bu^{k,n}_{\sigma,\lambda})\|_{\bL^2(\sO)}\right)^{}) \\
		&\leq \bE[(\Delta t)\sum_{n=0}^{N}
		\|	\bD^{\eta^n_*}\bu^{n+1}\|_{\bL^2(\sO)}\left( {(\Delta t)} \sum_{k=n-j+1}^n\|\bD^{\bfeta^{k-1}_*} \bu^{k}\|_{\bL^2(\sO)}\right) ]\\
		&\leq \sqrt{hT}\bE[(\Delta t)\sum_{n=0}^{N}
		\|	\bD^{\eta^n_*}\bu^{n+1}\|^2_{\bL^2(\sO)}] \leq Ch^{\frac12}.
	\end{align*}
	
Now for $I_6:=|	(\Delta t)\sum_{n=0}^N b^{\bfeta^n_*}\left( \bu^n,\bw_*^n,\left( \Delta t \sum_{k=n-j+1}^n\bu^{k,n}_{\sigma,\lambda}\right) \right) |$, the embedding $H^\frac12(\sO)\hookrightarrow L^4(\sO)$ gives us, for some $C>0$ that depends only on $\delta$, that
\begin{align*}
I_6^1&:=|(\Delta t)\sum_{n=0}^N	(\Delta t)\int_{\sO}J_*^n(\bu^{n+1}-
	\bw_*^{{n}})\cdot\nabla^{{\bfeta}_*^n}\bu^{n+1}\cdot\left( \Delta t \sum_{k=n-j+1}^n\bu^{k,n}_{\sigma,\lambda}\right)|\\
		&\leq\sqrt{h}	(\Delta t)\sum_{n=0}^N  \|\bu^{n+1}-\bw_*^{n}\|_{\bL^4(\sO)} \|\nabla^{\eta^n_*}\bu^{n+1}\|_{\bL^2(\sO)} \left( \Delta t \sum_{k=n-j+1}^n\|\bu^{k,n}_{\sigma,\lambda}\|^2_{\bL^4\sO)}\right)^{\frac12}\\
			&\leq { C(\delta)}\sqrt{h}\left( 	(\Delta t)\sum_{n=0}^N  \|\bu^{n+1}\|^2_{\bH^1(\sO)}\right)^{\frac32} \\
			&+C(\delta)\sqrt{h}\left((\Delta t) \sum_{n=0}^{N}\|\bv^{n+\frac12}\|^2_{\bL^2(0,L)}\right)^{\frac12} \left( (\Delta t) \sum_{n=0}^N\|\bu^{n+1}\|^2_{\bH^1(\sO)}\right),
\end{align*}
and similarly, using the property of mollification $\|\phi_\lambda\|_{H^m} \leq \frac{C}{\lambda^{m}}\|\phi\|_{L^2}$, we obtain
\begin{align*}
I_6^2&:=|	(\Delta t)\sum_{n=0}^N\int_{\sO}J^n_*(\bu^{n+1}-
	\bw_*^{{n}})\cdot\nabla^{{\bfeta}_*^n}\left( \Delta t \sum_{k=n-j+1}^n\bu^{k,n}_{\sigma,\lambda}\right)\cdot\bu^{n+1}|\\
		&\leq\sqrt{h}	(\Delta t)\sum_{n=0}^N  \|\bu^{n+1}-\bw_*^{n}\|_{\bL^2(\sO)} \|\bu^{n+1}\|_{\bL^4(\sO)} \left( \Delta t \sum_{k=n-j+1}^n\|\nabla^{\eta^n_*}\bu^{k,n}_{\sigma,\lambda}\|^2_{\bL^4(\sO)}\right)^{\frac12}\\
			& \leq\sqrt{h}	(\Delta t)\sum_{n=0}^N  \left( \|\bu^{n+1}\|_{\bL^2(\sO)}+\|\bv^{n+\frac12}\|_{\bL^2(0,L)}\right)  \|\bu^{n+1}\|_{\bH^1(\sO)} \left( \Delta t \sum_{k=n-j+1}^n
		\|\bu^{k}_{\sigma,\lambda}\|^2_{\bH^2(\sO_{\bfeta^n_*})}\right)^{\frac12}\\
		&\leq \frac{\sqrt{hT}}{{\lambda}}	\sup_{0\leq n\leq N}  \left( \|\bu^{n+1}\|_{\bL^2(\sO)}+\|\bv^{n+\frac12}\|_{\bL^2(0,L)}\right)(\Delta t)\sum_{n=0}^N  \|\bu^{n+1}\|^2_{\bH^1(\sO)}.
\end{align*}
Since $\lambda=h^{\frac16}$ for any $M>0$, we have
	\begin{align*}
		\bP\left( |	I^1_6 | \geq M^{}\right) &\leq \frac{h^{\frac13}}{M^{\frac23}}\bE\left( (\Delta t)\sum_{n=0}^N  \|\bu^{n+1}\|^2_{\bH^1(\sO)}+\|\bv^{n+\frac12}\|^2_{\bH^\frac12(0,L)}\right)\leq C\frac{h^{\frac13}}{\ep M^{\frac23}}\\
	\notag		\bP\left( |	I^2_6 | \geq M^{}\right) &\leq \frac{h^{\frac16}}{\sqrt{M}}\bE  \sup_{0\leq n\leq N}  \left( \|\bu^{n+1}\|^2_{\bL^2(\sO)}+\|\bv^{n+\frac12}\|^2_{\bL^2(0,L)}\right)\\
	\notag&+\frac{h^{\frac16}}{\sqrt{M}}\bE\sum_{n=0}^N (\Delta t) \|\bu^{n+1}\|^2_{\bH^1(\sO)}  
		 \leq C\frac{h^{\frac16}}{\sqrt{M}}. 
	\end{align*}
Next, calculations similar to \eqref{boundJt} give us,
\begin{align}\label{boundJt2}
	\|\frac{J^{n+1}_*-J^{n}_*}{\Delta t}\|_{\bL^2}= C(\delta)\|\bw^{n+1}_*\|_{\bH^1(\sO)}\leq C(\delta)\|\bv_*^{n+1}\|_{\bH^{\frac12}(0,L)}.
\end{align}
Hence,
	\begin{align*}
		I_7&:=	|\sum_{n=0}^N\int_\sO \left( J_*^{n+1}-J_*^n\right)  \bu^{n+1}\cdot \left( \Delta t \sum_{k=n-j+1}^n\bu^{k,n}_{\sigma,\lambda}\right) d\bx | \\
		&\leq (\Delta t)\sum_{n=0}^N\|\bv^{n+\frac12}\|_{\bH^\frac12(0,L)}\|\bu^{n+1}\|_{\bL^6(\sO)}\left( (\Delta t)\sum_{k=n-j+1}^n    \|  \bu^{k,n}_{\sigma,\lambda} \|^2_{\bL^3(\sO)}\right) ^{\frac12}\sqrt{h}\\
		&\leq C\sqrt{h} \left((\Delta t) \sum_{n=0}^N\|\bv^{n+\frac12}\|^2_{\bH^\frac12(0,L)} \right) ^{\frac12} \left( (\Delta t)\sum_{n=0}^N    \|  \bu^{n} \|^2_{\bH^1(\sO)}\right).
	\end{align*}
	This implies that,
	\begin{align}\label{ep2}
		\bP\left( |	I_7 | \geq M^{}\right) \leq \frac{h^{\frac14}}{M^{\frac12}}\bE\left((\Delta t)\sum_{n=0}^N\|\bv^{n+\frac12}\|^2_{\bH^\frac12(0,L)}+ (\Delta t)\sum_{n=0}^N  \|\bu^{n+1}\|^2_{\bH^1(\sO)}\right)\leq  \frac{C}{\ep}\frac{h^{\frac14}}{M^{\frac12}}.
	\end{align}
	Next, using \eqref{vh2} and the fact that $\|\phi_\lambda\|_{H^m} \leq \frac{C}{\lambda^{m}}\|\phi\|_{L^2}$, we obtain that,
	\begin{align*}
		I_8&:=	|(\Delta t)\sum_{n=0}^N	\left(   \sL_e(\bfeta^{n+1} ), \left( \Delta t \sum_{k=n-j+1}^n\bv^{k}_{\sigma,\lambda}\right)\right) |\\
		&\leq (\Delta t)^2 \sum_{n=0}^N\left( \|\partial_{zz}\bfeta^n \|_{\bL^2(0,L)} \sum_{k=n-j-1}^n\|\partial_{zz}\bv^{k,n}_{\sigma,\lambda}\|_{\bL^2(0,L)} \right)  \\
		& \leq (\Delta t)^2\sum_{n=0}^N\left( \|\partial_{zz}\bfeta^n \|_{\bL^2(0,L)} \sum_{k=n-j-1}^nC(\delta)\|\bfeta^n_*\|_{\bH^2(0,L)}\|\bu^{k}_{\sigma,\lambda}\|_{\bH^3(\sO_{\delta})} \right)\\
		& \leq \frac{C(\delta)}{\lambda^2}(\Delta t)^2 \sum_{n=0}^N\left(  \|\bfeta^n\|_{\bH^2(0,L)}\|\bfeta^n_*\|_{\bH^2(0,L)}\sum_{k=n-j-1}^n\|\bu^k_{\sigma}\|_{\bH^1(\sO_{\delta})} \right)\\
		&\leq {CT}{h^{\frac16}}\max_{0\leq n\leq N} \|\bfeta^n \|^2_{\bH^2(0,L)} \left( \sum_{n=1}^N\|\bu^n\|^2_{\bH^1(\sO)}\right) ^{\frac12}.
	\end{align*}
Here we also used that $\lambda = h^{\frac16}$.	Hence for $I_8$ we obtain,
	\begin{align*}
		\bP(I_8 \geq M)  &\leq \frac{Ch^{\frac1{12}}}{M}  \bE\left( \max_{0\leq n\leq N}\|\bfeta^n \|^2_{\bH^2(0,L)}+ \max_{0\leq n\leq N}\|\bu^n\|^2_{\bL^2(\sO)}\right) 
		\leq C\frac{h^{\frac1{12}}}{M} .
	\end{align*}
Similarly, for some $C>0$, independent of $\ep$ and $N$, we obtain
\begin{align*}
&\bE|\ep	(\Delta t) \sum_{n=0}^N \int_0^L\partial_z\bv^{n+\frac12} \partial_z \left( \Delta t \sum_{k=n-j+1}^n\bv^{k,n}_{\sigma,\lambda}\right) dz|\\
& \leq\bE \left( \ep	(\Delta t) \sum_{n=0}^N \|\partial_z\bv^{n+\frac12}\|_{\bL^2(0,L)} \left( \Delta t \sum_{k=n-j+1}^n\|\partial_z \bv^{k,n}_{\sigma,\lambda}\|_{\bL^2(0,L)}\right)\right)  \\
& \leq\bE\left(  \ep	(\Delta t) \sum_{n=0}^N \|\partial_z\bv^{n+\frac12}\|_{\bL^2(0,L)} \left( \Delta t \sum_{k=n-j+1}^n\|\nabla \bu^{k}_{\sigma,\lambda}\|_{\bH^\frac12(\sO_{\bfeta^n_*})}\right)\right)  \\
& \leq\bE\left[ \sqrt{T}\frac{h}{\lambda^\frac32}\left( \ep	(\Delta t) \sum_{n=0}^N \|\partial_z\bv^{n+\frac12}\|^2_{\bL^2(0,L)}\right)^{\frac12}  \left(\sup_{0\leq k\leq N}\| \bu^{k}_{\sigma}\|_{\bL^2(\sO_{\delta})}\right) \right] \\
& \leq h^{\frac12} C(\delta, T) \bE\left[ \left( \ep	(\Delta t) \sum_{n=0}^N \|\partial_z\bv^{n+\frac12}\|^2_{\bL^2(0,L)}\right)+  \left(\sup_{0\leq k\leq N}\| \bu^{k}_{}\|^2_{\bL^2(\sO_{})}\right) \right] \leq C h^{\frac12}
\end{align*}
	Finally, we treat the stochastic term using the same argument argument as above. To bound the expectation we use Young's inequality and the argument presented in \eqref{tower} to obtain,
	\begin{align*}
		&\bE\left( \sum_{n=0}^N	|(G(\bU^{n},\bfeta_*^n)\Delta_n W, \bQ_n)|\right) \\
		&	\leq 		\bE\left( \sum_{n=0}^N	\|G(\bU^{n},\bfeta_*^n)\|_{L_2(U_0;\bL^2)}\|\Delta_n W\|_{U_0}  \left( (\Delta t) \sum_{k=n-j+1}^n\|\bu^{k,n}_{\sigma,\lambda}\|_{\bL^2(\sO)}^2 +\|\bv_{\sigma,\lambda}^k\|_{\bL^2(0,L)}^2\right)^{\frac12}h^{\frac12}\right) 	\\
		& 	\leq 	h^{\frac12}	\bE\left( \sum_{n=0}^N	\|G(\bU^{n},\bfeta_*^n)\|^2_{L_2(U_0;\bL^2)}\|\Delta_n W\|^2_{U_0}  +\left( (\Delta t) \sum_{k=n-j+1}^n\|\bu^{k,n}_{\sigma,\lambda}\|_{\bL^2(\sO)}^2 +\|\bv_{\sigma,\lambda}^k\|_{\bL^2(0,L)}^2\right)\right) 	\\
		&\leq h^{\frac12} C(\delta,\Tr \bQ)\bE\sum_{n=0}^N(\Delta t)[\|\bu^n\|^2_{\bL^2(\sO)}+\|\bv^n\|^2_{\bL^2(0,L)}] \leq Ch^{\frac12}.
	\end{align*}
	Now to show that the laws of the random variables mentioned in the statement of the theorem are tight, we will consider the following sets for $0\leq \alpha<1$ and $0\leq \beta<\frac12$ and any $M>0$,
	\begin{align*}
		{\mathcal{B}}_M:=\{(\bu,\bv)&\in L^2(0,T;\bH^\alpha(\sO))\times L^2(0,T;\bH^{\beta}(0,L)):\|{\bu}\|^2_{L^2(0,T;\bH^1(\sO))}+\|{\bv}\|^2_{L^2(0,T;\bH^{\frac12}(0,L))}\\
		&+\sup_{0<h<1}{h^{-\frac1{18}}}\int_h^{T}\left( \|T_h\bu-\bu\|^2_{\bL^2(\sO)}+\|T_h\bv-\bv\|^2_{\bL^2(0,L)}\right)  \le M\}.
	\end{align*}
	Observe that, thanks to Lemma \ref{compactLp}, $\sB_M$ is compact in $L^2(0,T;\bH^\alpha(\sO))\times L^2(0,T;\bH^{{\beta}}(0,L))$, $0\leq \alpha<1$ and $0\leq\beta<\frac12$ for each $M>0$.
	An application of Chebyshev's inequality then gives us the desired result:
	\begin{align*}
		\bP((\bu^+_N,\bv^+_N) \notin \sB_M)&\leq \bP\left( \|{\bu^+_N}\|^2_{L^2(0,T;\bH^1(\sO))}+\|{\bv^+_N}\|^2_{L^2(0,T;\bH^{\frac12}(0,L))}>\frac{M}2\right) \\
		&+\bP\left( \sup_{0 <h<1 }{h^{-\frac1{18}}}\int_h^{T}\left( \|T_h\bu^+_N-\bu^+_N\|^2_{\bL^2(\sO)}+\|T_h\bv^+_N-\bv^+_N\|^2_{\bL^2(0,L)}\right)  > \frac{M}2\right) \\
		&\leq \frac{C}{\sqrt{M}},
	\end{align*}
	where $C>0$ depends only on $\delta$, Tr$Q$, $\ep$ and the given data and is independent of $N$. Note here that the dependence of $C$ on $\ep$ appears only in \eqref{ep2}. We shall see in the next section that, by integrating by parts, we can get rid of this dependence and obtain the same tightness result which will allow us to pass $\ep \to 0$.
\end{proof}
Next we will state the rest of the tightness results. These will be used in Section \ref{almostsure} to obtain almost sure convergence via an application of Prohorov's theorem and the Skorohod representation theorem. 
\begin{lem}\label{tightl2}
	For fixed a $\delta$ and $\ep>0$, the following statements hold:
	\begin{enumerate}
		\item The laws of $\{\tilde\bfeta_{N}\}_{N\in \mathbb{N}}$ and that of $\{\tilde\bfeta^*_{N}\}_{N\in \mathbb{N}}$ are tight in $C([0,T],\bH^{s}(0,L))$ for any $s<2$.
		\item The laws of $\{\|\bv^*_N\|_{L^2(0,T;\bH^1_0(0,L))}\}_{N\in \mathbb{N}}$ are tight in $\R$.
		\item The laws of $\{\|\bu^+_N\|_{L^2(0,T;V)}\}_{N\in \mathbb{N}}$ are tight in $\R$.
	\end{enumerate}
\end{lem}
	\begin{proof}
		To prove the first statement we observe, thanks to Lemma \ref{bounds}, that $\tilde{\bfeta}_{N}$ and $\tilde{\bfeta}_{N}^*$ are bounded independently of $N$ in $L^2(\Omega;L^\infty(0,T;\bH^2(0,L))\cap W^{1,\infty}(0,T;\bL^2(0,L)))$. A direct application of the Aubin-Lions theorem gives us that for $0<s<2$
		$$L^\infty(0,T;\bH^2(0,L))\cap W^{1,\infty}(0,T;\bL^2(0,L)) \subset\subset C([0,T];\bH^{s}(0,L)). 
		$$
		Hence consider
		\begin{align*}\mathcal{K}_M:=&\{{\bfeta}\in L^\infty(0,T;\bH^2_0(0,L))\cap W^{1,\infty}(0,T;\bL^2(0,L)):\\
			&\|{\bfeta}\|^2_{L^\infty(0,T;\bH^2_0(0,L))} + \|{\bfeta}\|^2_{ W^{1,\infty}(0,T;\bL^2(0,L))}\leq M\}.\end{align*}
		Using the Chebyshev inequality once again we obtain for some $C>0$ independent of $N$ that the following holds:
		\begin{equation}
			\begin{aligned}
				\bP\left[\tilde{\bfeta}_{N}\not\in {\mathcal{K}}_M\right]&\le \bP\left[
				{\|\tilde{\bfeta}_{N}\|}^2_{L^\infty(0,T;\bH^2_0(0,L))}\ge \frac{M}{2}\right]+
				\bP\left[
				{\|\tilde{\bfeta}_{N}\|}^2_{W^{1,\infty}(0,T;\bL^2(0,L))}\ge \frac{M}{2}\right]	\\
				&\le  \frac{4}{M^2}\bE\left[
				{\|\tilde{\bfeta}_{N}\|}^2_{L^\infty(0,T;\bH^2_0(0,L))}+\|\tilde{\bfeta}_{N}\|^2_
				{W^{1,\infty}(0,T;\bL^2(0,L))}\right]
				\leq \frac{C}{M^2}.
			\end{aligned}
		\end{equation}
	The proof of the statements (2) follow by a similar application of the Chebyshev inequality using the bounds obtained in Lemma \ref{bounds}. For any $M>0$ and fixed $\ep>0$,
	\begin{align*}
		\bP[\| \bv^*_N\|_{L^2(0,T;\bH^1_0(0,L))}>M] \leq \frac1{M^2}\bE[\| \bv^*_N\|^2_{L^2(0,T;\bH^1_0(0,L))}] \leq \frac{C}{M^2}.
	\end{align*}
	{This completes the proof of the tightness results stated in Lemma \ref{tightl2}.}
	\end{proof}
	To obtain almost sure convergence of the rest of the random variables we will use the following lemma.
	\begin{lem}\label{difference}
	The following convergence results hold: 
		\begin{enumerate}
			\item $\lim_{N\rightarrow\infty} \bE\int_0^T\|\bu_{N}-\bu^+_{N}\|^2_{\bL^2(\sO)}dt=0$,
			\quad $\lim_{N\rightarrow\infty} \bE\int_0^T\|\bu_{N}-\tilde\bu_{N}\|^2_{\bL^2(\sO)}dt=0$.
				\item  $\lim_{N\rightarrow\infty} \bE\int_0^T\|\bv_{N}-\tilde \bv_{N}\|^2_{\bL^2(0,L)}dt=0$,
			\quad $\lim_{N\rightarrow\infty} \bE\int_0^T\|\bv_{N}-\bv^{\#}_{N}\|^2_{\bL^2(0,L)}dt=0$.
			\item $\lim_{N\rightarrow\infty} \bE\int_0^T\|{\bfeta}_{N}-\tilde{\bfeta}_{N}\|^2_{\bH_0^2(0,L)}dt=0$,
			\quad $\lim_{N\rightarrow\infty} \bE\int_0^T\|{\bfeta}^+_{N}-\tilde{\bfeta}_{N}\|^2_{\bH_0^2(0,L)}dt=0$.
			\item $\lim_{N\rightarrow\infty} \bE\int_0^T\|{\bfeta}^*_{N}-\tilde{\bfeta}^*_{N}\|^2_{\bH_0^2(0,L)}dt=0$.		
		\end{enumerate}
	\end{lem}
	\begin{proof} Statement (1)$_1$ follows immediately from Theorem \ref{energythm} (3). We prove (1)$_2$ below. The rest follows similarly {\s from the uniform estimates stated in} Theorem \ref{energythm}.
		\begin{align*}
		&	\bE\int_0^T\|\bu_{N}-\tilde\bu_{N}\|^2_{\bL^2(\sO)}dt=
			\bE\sum_{n=0}^{N-1}\int_{t^n}^{t^{n+1}}\frac1{\Delta t}\|(t-t^n)\bu^{n+1}+(t^{n+1}-t-\Delta t)\bu^{n}\|^2_{\bL^2(\sO)}dt\\
			&=\bE\sum_{n=0}^{N-1}\|\bu^{n+1}-\bu^n\|_{\bL^2(\sO)}^2\int_{t^n}^{t^{n+1}} \left( \frac{t-t^n}{\Delta t}\right) ^2dt
			\leq \frac{CT}{\delta_1 N} \rightarrow 0 \quad N\rightarrow \infty.
		\end{align*}
	\end{proof}
To pass $N\to\infty$  in the weak formulation, we consider the following random variable:
	\begin{align}\label{UN}
		\tilde\bU_N(t):=(	(\tilde J_N^*(t))\tilde\bu_N(t),\tilde \bv_N(t))-E_N(t),
	\end{align}
	where $E_N$ is an error term that appears due to discretizing the stochastic integral (see \eqref{Jform})
	$$E_N(t)=\sum_{m=0}^{N-1}\left( \frac{t-t^m}{\Delta t}G(\bu^m,\bv^m,\bfeta^m_*)\Delta_mW-\int_{t^m}^tG(\bu^m,\bv^m,\bfeta^m_*)dW\right) \chi_{[t^m,t^{m+1})}.$$
	Next we define $	\sU_1:=\sU\cap (\bH^2(\sO)\times \bH^3_0(0,L)).$
For any $\sV_1\subset\subset \sU_1$, we denote by $\mu^{u,v}_N$ the probability measure of $\tilde\bU_N$: $$\mu^{u,v}_{N}:=\bP\left( \tilde\bU_{N}\in \cdot \right) \in Pr(C([0,T];\sV'_1)) ,$$
	where $Pr(S)$, here and onwards, denotes the set of probability measures on a metric space $S$. Then we have the following tightness result proven in \cite{TC23}.
	\begin{lem}\label{uvtightC}
		For fixed $\ep>0$ and $\delta$, the laws $\{\mu_N^{u,v}\}_N$ of {the random variables } $\{\tilde\bU_N\}_{N}$ 
		are tight in 
		$C([0,T];\sV_1')$.
			\end{lem}
			Moreover, in the following lemma, we show that the error term vanishes as $N\to \infty$.
			\begin{lem}\label{u*diff}
				{The numerical error $E_N$ of the stochastic term has the following property:}
				$$ {\bE\int_0^T\|E_N(t)\|^2_{\bL^2(\sO)\times \bL^2(0,L)}dt  \rightarrow } 0 \text{ as } {N\rightarrow\infty}.$$
			\end{lem}
			\begin{proof}
				First, for any $N$, we have 
				$$E_N(t)=\sum_{m=0}^{N-1}\left( \frac{t-t^m}{\Delta t}G(\bu^m,\bv^m,\bfeta^m_*)\Delta_mW-\int_{t^m}^tG(\bu^m,\bv^m,\bfeta^m_*)dW\right) \chi_{[t^m,t^{m+1})}=:E_N^1+E_N^2$$
				{We estimate $E^1_N$ and $E^2_N$. Recall our notation $\bL^2=\bL^2(\sO)\times \bL^2(0,L)$.	Observe that $E^1_N$ satisfies (see also \eqref{tower})},
				\begin{align*}
					&\bE\int_0^T\|E_N^1(t)\|_{\bL^2}^2dt	=\bE\sum_{n=0}^{N-1} \|G(\bu^n,\bv^n,\bfeta^n_*)\Delta_nW\|_{\bL^2}^2 \int_{t^n}^{t^{n+1}}|\frac{t-t^m}{\Delta t}|^2dt\\
					&=\bE\sum_{n=0}^{N-1} \|G(\bu^n,\bv^n,\bfeta^n_*)\Delta_nW\|_{\bL^2}^2 \frac{\Delta t}{3 }\leq \bE\sum_{n=0}^{N-1} \|G(\bu^n,\bv^n,\bfeta^n_*)\|_{L_2(U_0;\bL^2)}^2\|\Delta_nW\|_{U_0}^2 {\Delta t}\\
					&	= (\Delta t)^2\bE\sum_{n=0}^{N-1} \|G(\bu^n,\bv^n,\bfeta^n_*)\|_{L_2(U_0;\bL^2)}^2 \leq C \Delta t,
				\end{align*}
				where $C>0$ does not depend on $N$ or $\ep$, as a consequence of Theorem \ref{energythm}.
				
				{To estimate $E^2_N$ we use the It\^{o} isometry to obtain}
				\begin{align*}
					&\bE\int_0^T\|E_N^2(t)\|_{\bL^2}^2dt= \bE\sum_{n=0}^{N-1}\int_{t^n}^{t^{n+1}}\|\int_{t^n}^tG(\bu^n,\bv^n,\bfeta_*^n)dW\|_{\bL^2}^2dt\\
					&= \bE\sum_{n=0}^{N-1}\int_{t^n}^{t^{n+1}}\int_{t^n}^t\|G(\bu^n,\bv^n,\bfeta_*^n)\|_{L_2(U_0;\bL^2)}^2dsdt
					\\&= \frac12\bE\sum_{n=0}^{N-1}\|G(\bu^n,\bv^n,\bfeta_*^n)\|_{L_2(U_0;\bL^2)}^2(\Delta t)^2 \leq C\Delta t.
				\end{align*}	
				{Thus}, as $N \rightarrow \infty$ we {obtain}
				$\bE\int_0^T\|E_N(t)\|^2_{\bL^2(\sO)\times L^2(0,L)}dt \rightarrow 0.$	
			\end{proof} 
	\subsection{Almost sure convergence}\label{almostsure}
	Let 
	$\mu_{N}$ be the joint law of the random variable\\
	$\mathcal{U}_{N}:=(\bu^+_{N},\bv^{\#}_{N},\bv^+_{N},\tilde\bfeta^*_{N},\tilde\bfeta_{N},\|\bu^+_N\|_{L^2(0,T;V)}, \|\bv^*_N\|_{L^2(0,T;\bH^1_0(0,L))},\tilde\bU,W)$ taking values in the phase space
	\begin{align*}
		\Upsilon&:=[L^2(0,T;\bL^{2}(\sO))]\times[L^2(0,T;\bL^2(0,L))]^2\times [C([0,T],\bH^s(0,L))]^2\times\R^2\\
		& \times L^2(0,T;\bL^2(\sO)\times L^2(0,L))\cap C([0,T];\sV_1')\times C([0,T];U),
	\end{align*}
	for some fixed $\frac32<s<2$.

	Since $C([0,T];U)$ is separable and metrizable by a complete metric,  the sequence of Borel probability measures, $\mu^W_{N}(\cdot):=\bP(W\in\cdot)$, that are constantly equal to one element, is tight on $C([0,T];U)$.
	
	Next, recalling Lemmas \ref{tightuv} \ref{tightl2}, \ref{u*diff} and using Tychonoff's theorem it follows that, for a fixed $\ep>0$, the sequence of the joint probability measures $\mu_{N}$ of the approximating sequence $\mathcal{U}_{N}$ is tight on the Polish space $\Upsilon$. Hence,  by applying the Prohorov theorem and the Skorohod representation theorem we obtain the following result. To be precise we use the version given in Theorem 1.10.4 in \cite{VW96}.
	\begin{theorem}\label{skorohod} 
		There exists a probability space $(\bar\Omega,\bar\sF,\bar\bP)$ and random variables \\$\bar{\mathcal{U}}_{N}:=(
		\bar{\bu}^+_{N},\bar \bv^{\#}_{N},\bar \bv^+_{N},\bar{\tilde\bfeta}^*_{N},\bar{\tilde\bfeta}_{N},m_N,k_N,\bar{\tilde\bU}_N,\bar{W}_N)$ and
		$\bar{\mathcal{U}}:=(
		\bar{\bu},{\bar \bv},\bar \bv^{},\bar{\bfeta}^*,\bar{\bfeta},m,k, \bar{\tilde\bU},\bar{W})$
		defined on this new probability space, such that
	\begin{align}\label{convu}
		\bar{\mathcal{U}}_{N} \rightarrow \bar{\mathcal{U}},\qquad\quad\bar\bP-a.s. \quad\text{in the topology of } \Upsilon.
	\end{align}
		Additionally, there exist measurable maps $\phi_N:\bar\Omega\rightarrow\Omega$ such that \begin{align}\label{newrv}
			\bar{\mathcal{U}}_N(\bar\omega)={\mathcal{U}}_N(\phi_N(\bar\omega)) \quad \text{ for }\bar\omega\in\bar\Omega,
		\end{align} and $\bar \bP\circ\phi_N^{-1}=\bP$ 
		
	\end{theorem}

	Now we will define,
	\begin{align*}
		\bar\bu_N=\bu_N\circ\phi_N,\quad 	\bar{\tilde\bu}_N=\tilde\bu_N\circ\phi_N,\\
		\bar{\bfeta}_N={\bfeta}_N\circ\phi_N,\quad	{\bar{\bfeta}}^+_N={\bfeta}^+_N\circ\phi_N, \quad	\bar{\bfeta}^*_N={\bfeta}^*_N\circ\phi_N,\\
		\bar \bv_N=\bv_N\circ\phi_N.
	\end{align*}
	Then, from part (1), Lemmas \ref{difference} and an application of the Borel-Cantelli lemma, we obtain
	\begin{align}
	&	\bar\bu_N \to \bar\bu,\quad 	\bar{\tilde\bu}_N \to \bar\bu, \quad \bar\bP-a.s. \quad\text{in } L^2(0,T;\bL^2(\sO)),\\	&\bar{\bfeta}_N\to\bar{\bfeta},\quad{\bar{\bfeta}}^+_N\to\bar{\bfeta}, \quad	\bar{\bfeta}^*_N \to\bar{\bfeta}^*, \quad \bar\bP-a.s. \quad\text{in } L^2(0,T;\bH^s(0,L)),\,\, \frac32<s<2.\\
	&	\bar \bv_N \to \bar \bv, \quad \bar\bP-a.s. \quad\text{ in } L^2(0,T;\bL^2(0,L)).\label{convv}
	\end{align}
		Thanks to these explicit maps we can identify the real-valued random variables ${k}_N= \|\bar \bv^*_N\|_{L^2(0,T;\bH^1_0(0,L))}$. Hence almost sure convergence of ${k}_N$ implies that $\|\bar \bv^*_N\|_{L^2(0,T;\bH^1_0(0,L))}$ is bounded a.s. and thus, up to a subsequence, 
		\begin{align}\label{vweak}
		\bar\bv^*_N \rightharpoonup \bar\bv^* \quad\text{ weakly in } {L^2(0,T;\bH^1_0(0,L))} \quad \bar\bP-a.s.
	\end{align}
	Similarly we have that $\bar\bu^+_N\rightharpoonup\bar\bu$ in $L^2(0,T;V)$ a.s.

	Furthermore, we have,
	\begin{align}\partial_t\bar{\bfeta}=\bar \bv \text{ and } \partial_t\bar{\bfeta}^*=\bar \bv^*,\quad \text{ in the sense of distributions, almost surely.}\end{align}
	This is true since we have
	$$
	-\int_0^T\int_0^L \tilde{\bfeta}_{N}\partial_t\bfphi dzdt=\int_0^T\int_0^L  { \bv^{\#}_{N} \bfphi} dzdt,\quad \forall \phi\in C_0^1(0,T;\bL^2(0,L)),$$
	which together with Theorem \ref{skorohod} implies,
	$$
	-\int_0^T\int_0^L \bar{\tilde{\bfeta}}_{N}\partial_t\bfphi dzdt=\int_0^T\int_0^L{ \bar{ \bv}^{\#}_{N} \bfphi }dzdt,\quad \forall \bfphi\in C_0^1(0,T;\bL^2(0,L)).$$
	Then passing $N \rightarrow \infty$ we come to the desired conclusion. The second half of the statement can be proven similarly.

		We also have the following upgraded convergence results for the displacements. Namely,
	notice that Theorem \ref{skorohod} implies that for given $\frac32<s<2$ (see \cite{MC13} Lemma 3),
	\begin{align}\label{etaunif1}
		\bar{\bfeta}_{N} \rightarrow \bar{\bfeta} \text{ and } \bar{\bfeta}^*_{N} \rightarrow \bar{\bfeta}^* \,\,\,\text { in } L^\infty(0,T;\bH^s(0,L))\, a.s.
	\end{align}
	and thus the following uniform convergence result holds
	\begin{align}\label{etaunif}
		\bar{\bfeta}_{N} \rightarrow \bar{\bfeta} \text{ and } \bar{\bfeta}^*_{N} \rightarrow \bar{\bfeta}^*\,\,\, \text { in } L^\infty(0,T;\bC^1[0,L])\, a.s.
	\end{align}
Next, we define $\bar\bu^n=\bu^n\circ \phi_N, \bar\bv^n=\bv^n\circ \phi_N$, $\bar{\bfeta}^n={\bfeta}^n\circ \phi_N$ and $\bar{\bfeta}_*^n={\bfeta}_*^n\circ \phi_N$ 
for all $n=1,2,..,N$. Using these notations, we define piecewise constant interpolations  $ J_{\bar\bfeta^*_N}$, $ A_{\bar\bfeta^*_N}$,  and the piecewise linear interpolation $\tilde { J}_{\bar\bfeta^*_N}, \tilde{A}_{\bar\bfeta^*_N}$. Then thanks to \eqref{etaunif}, \eqref{boundJ} and \eqref{boundsA1}, we have
\begin{equation}\label{convrest}
	\begin{split}
		& \nabla A_{\bar\bfeta^*_N} \rightarrow \nabla A_{\bar\bfeta^*} \text{ and } (\nabla A_{\bar\bfeta^*_N})^{-1} \rightarrow (\nabla A_{\bar\bfeta^*})^{-1} \quad \text{ in } L^\infty(0,T;\bC(\sO))\, a.s.\\
		&			 J^*_{\bar\bfeta^*_N} \rightarrow  J_{\bar\bfeta^*} =\text{det}(\nabla A_{\bar\bfeta^*})\quad \text{ in } L^\infty(0,T;C(\bar\sO))\, a.s.
	\end{split}
\end{equation}
Furthermore, $A_{\bar\bfeta^*} \in L^\infty(0,T;\bW^{2,p}(\sO)$ for $p<4$ and is the solution to \eqref{ale} for boundary data $\text{\bf id} + \bar\bfeta^*$ on $\Gamma$.

Next let $\bar\bw^*_N=\partial_t \tilde{A}_{\bar\bfeta^*_N}=\sum_{n=0}^N\frac{1}{\Delta t}(A^\omega_{\bar{\bfeta}_*^{n}}-A^\omega_{\bar{\bfeta}^{n-1}_*})\chi_{(t_n,t_{n+1})}$. Observe that, due to \eqref{ale}, for every $\omega\in\bar\Omega$, we have (see \cite{G11}):
\begin{align*}
\|\bw^*_N
\|_{L^2(0,T;\bH^{s+\frac12}(\sO))}\leq C \|\bar\bv^{*}_N
\|_{L^2(0,T;\bH^s(0,L))},\quad \text{ for any } 0\leq s\leq  \frac12.
\end{align*}
 Thus using \eqref{vweak} we obtain,
\begin{align}\label{convw}
	\bar\bw^*_N \rightharpoonup \bar\bw^*\, \quad \text{ weakly in } L^2(0,T;\bH^{1}(\sO)),\quad \, \bar\bP-a.s,
\end{align}
where $\bar\bw^*$ satisfies \eqref{ale} with boundary values $\bar\bv^*$. Similarly, \eqref{boundJt2} gives us
\begin{align*}
	\partial_t \tilde J_{\bar\bfeta^*_N} \rightharpoonup 	\partial_t  J_{\bar\bfeta^*} \quad \text{ weakly in } L^2(0,T;\bL^{2}(\sO)),\quad \, \bar\bP-a.s.
\end{align*}

	Finally, we give the definition of the required filtrations. we denote by $\sF_t'$ the $\sigma$-field generated by the random variables $(\bar{\bu}(s),\bar \bv(s)),\bar\bfeta(s),\bar{W}(s)$ for all $s \leq t$. Then we define
	\begin{align}\label{Ft}
		\mathcal{N} &:=\{\mathcal{A}\in \bar{\mathcal{F}} \ |  \bar \bP(\mathcal{A})=0\},\qquad
		\bar{\mathcal{F}}^0_t:=\sigma(\mathcal{F}_t' \cup \mathcal{N}),\qquad
		\bar{\mathcal{F}}_t :=\bigcap_{s\ge t}\bar{\mathcal{F}}^0_s.
	\end{align}
	We note here that Lemma \ref{uvtightC} and \eqref{etaunif} give us stochastic processes $(J_{\bar\bfeta^*}\bar\bu,\bar \bv)$  that are $(\bar\sF_t)_{t\geq 0}-$progressively measurable and thus helps in identifying the limit of the stochastic integral as we pass to the limit as $N \rightarrow\infty$. 
	
	For each $N$ we also define a filtration $\left(\bar{\mathcal{F}}^{N}_t \right)_{t \geq 0}$ on  $(\bar \Omega,\bar{\mathcal{F}},\bar \bP)$ the same way as above but using the processes $(\bar{\bu}_{N}(s),\bar{\bv}_{N}(s)),\bar\bfeta_N(s),\bar W_N(s)$  instead. Because of \eqref{newrv}, we can see that this pointwise definition of the filtration $(\bar\sF^N_t)_{t\geq 0}$ makes sense.  Moreover, using usual arguments we can see that $\bar W_N$ is an $\bar\sF^N_t$-Wiener process (see e.g. \cite{BFH18}).
Next, relative to the new stochastic basis $(\bar\Omega,\bar\sF,(\bar\sF^{N}_t)_{t\geq0},\bar\bP,\bar W_N)$, 
	thanks to \eqref{newrv} we can see that for each $N$, the following equation holds $\bar\bP$-a.s. for every $t \in [0,T]$ and any $\bQ\in \sD\cap \sU_1$:
	\begin{equation}\begin{split}\label{approxsystem}
			&(\bar{\tilde\bU}_N,\bQ)=
			(\bu_0J_0,\bq)+( \bv_0,\bfpsi)	-\int_0^t\langle \sL_e(\bar\bfeta^+_N),\bfpsi \rangle -\ep\int_0^t\int_0^L\partial_z\bar\bv^{\#}_N\cdot\partial_z\bfpsi  \\
			&-\frac{1}2\int_0^t\int_\sO 
			\partial_t\tilde{ J}_{\bar\bfeta^*_N} \bar\bu^{+}_{N}\cdot \bq  +\int_0^t\int_{\sO}
			\partial_t \tilde{ J}_{\bar\bfeta^*_N} (2\bar{\tilde\bu}_{N}-\bar{\bu}_{N})\cdot\bq \\
			&-\frac12\int_0^t\int_{\sO} J_{\bar\bfeta^*_N}((\bar\bu^+_{N}-\bar \bw^{*}_{N} 
			)\cdot\nabla^{\bar\bfeta_{N}^*}\bar\bu^{+}_{N}\cdot\bq
			- (\bar\bu^{+}_{N}-
			\bar{\bw}^{*}_{N} 
			)\cdot\nabla^{\bar\bfeta_{N}^*}\bq\cdot\bar\bu^{+}_{N})\\
			&-2\nu\int_0^t\int_{\sO}(J_{\bar\bfeta^*_N})\bD^{\bar\bfeta_{N}^*}(\bar\bu^{+}_{N})\cdot \bD^{\bar\bfeta_{N}^*}(\bq) dxds -\frac1{\ep}\int_0^t \int_\sO   \text{div}^{\bar\bfeta^*_{N}}\bar\bu^{+}_{N}\text{div}^{\bar\bfeta^*_{N}}\bq \\
			&+\int_0^t\left( P_{{in}}\int_{0}^1q_z\Big|_{z=0}-P_{{out}}\int_{0}^1q_z\Big|_{z=1}\right)
			+\int_0^t 
			(G({\bar\bu}_{N},\bar \bv_N,\bar\bfeta^*_{N})d\bar W_N, \bQ).
		\end{split}
	\end{equation}
Using the convergence results stated in Theorem \ref{skorohod} and a density argument, we can pass $N \to \infty$ in \eqref{approxsystem} to arrive at the following theorem. 	The proof of the following result is similar to and simpler than the proof of Theorem \ref{exist2} and thus we refer the reader to the proof of Theorem \ref{exist2} for further details. We do mention here how we deal with the convective term. By integrating by parts we obtain
\begin{equation}\begin{split}\label{byparts}
	\int_\sO  J_{\bar\bfeta^*_N} \bar\bw_N^*\cdot\nabla^{\bar\bfeta_N^*}\bar\bu^+_N\cdot\bq&=-\int_{\sO} {J}_{\bar\bfeta^*_N} \nabla^{\bar\bfeta^*_N}\cdot\bar\bw_N^*\bar{\bu}^+_N \cdot\bq-\int_{\sO} {J}_{\bar\bfeta^*_N} \bar\bw^*_N\cdot \nabla^{\bar\bfeta^*_N}\bq \cdot \bar{\bu}^+_N\\
	&+\int_\Gamma S_{\bar\bfeta_N^*}(\bar\bv_N^*\cdot\bn_N^*)(\bar\bu^+_N\cdot\bq),
\end{split}	\end{equation}
	where $S_{\bar\bfeta^*_N}$ is the Jacobian of the transformation from Eulerian to Lagrangian coordinates.
Thus, using the weak and strong convergence results \eqref{convu}, \eqref{vweak}, \eqref{convw} and \eqref{etaunif} we can pass $N\to \infty$ in the terms above.

	We are now ready to state the main result of this section.
	The following theorem establishes the desired existence result for the approximate system containing the penalty term.
		\begin{theorem}[Existence for the problem with penalized compressibility]\label{exist1} 
		For the stochastic basis $(\bar\Omega,\bar\sF,(\bar\sF_t)_{t \geq 0},\bar\bP,\bar W)$ as found in Theorem \ref{skorohod}, given any fixed $\ep>0$ and $\delta=(\delta_1,\delta_2)$ satisfying \eqref{delta}, the processes $(\bar\bu,\bar\bfeta,\bar\bfeta^*)$ obtained in Theorem \ref{skorohod} are such that $\bar\bfeta^*$ and $(\bar J_{\bfeta^*}\bar\bu,\partial_t\bar\bfeta)$ are $(\bar\sF_t)_{t \geq 0}$-progressively measurable with $\bar\bP$-a.s. continuous paths in $\bH^s(0,L)$, $\frac32<s<2$ and $\sV_1'$ respectively {and such that the following weak formulation 
			holds $\bar\bP$-a.s. for every $t\in[0,T]$ and for every $\bQ\in \sD$}:
		\begin{equation}\begin{split}\label{martingale1}
				&(	J_{\bar\bfeta^*}(t)\bar\bu (t),\bq)+(\partial_t\bar{ \bfeta}(t),\bfpsi)
				=(\bu_0(J_0),\bq)+( \bv_0,\bfpsi)\\
				&-\int_0^t\langle \sL_e(\bar\bfeta),\bfpsi \rangle -\ep\int_0^t\int_0^L\partial_z\partial_t\bar\bfeta\cdot\partial_z\bfpsi  +\frac1{\ep}\int_0^t \int_\sO   \text{div}^{\bar\bfeta^*}\bar\bu\text{ div}^{\bar\bfeta^*}\bq\\
				&-\frac12\int_0^t\int_{\sO}(J_{\bar\bfeta^*})(\bar\bu\cdot\nabla^{\bar\eta ^*}\bar\bu \cdot\bq
				- (\bar\bu^{}-2\bar\bw^*)\cdot\nabla^{\bar\eta ^*}\bq\cdot\bar\bu ) +\frac12\int_0^t\int_\Gamma S_{\bar\bfeta^*}(\partial_t\bar\bfeta^*\cdot\bn^*)(\partial_t\bar\bfeta\cdot\bfpsi)				\\
				&-2\nu\int_0^t\int_{\sO}J_{\bar\bfeta^*}\bD^{\bar\eta ^*}(\bar\bu )\cdot \bD^{\bar\eta ^*}(\bq)  +\int_0^t\left( P_{{in}}\int_{0}^1q_z\Big|_{z=0}-P_{{out}}\int_{0}^1q_z\Big|_{z=1}\right)\\
				&+\int_0^t
				(	G(\bar\bu,\partial_t\bar\bfeta ,\bar\bfeta^*)d\bar W, \bQ),
			\end{split}
		\end{equation}
		where $\bar\bw^*=\partial_tA_{\bar\bfeta^*}$, $J_{\bar\bfeta^*}=det\nabla A_{\bar\bfeta^*}$ and $\bn^*$ is the unit normal to $\Gamma_{\bar\bfeta^*}$.
	\end{theorem}
Before proceeding we make a few comments. First, observe that, to arrive at the form \eqref{martingale1}, we used the fact that
\begin{align}\label{Jform}
	\int_0^t\int_{\sO}\partial_tJ_{\bar\bfeta^*}\bar{\bu} \cdot\bq=\int_0^t\int_{\sO}J_{\bar\bfeta^*} \nabla^{\bfeta^*}\cdot\bw^*\bar{\bu} \cdot\bq.
\end{align}

		Next, we argue that 
		\begin{align}\label{etasequal1}
			\bar{\bfeta}^*(t)=\bar{\bfeta}(t) \quad \text{ for any } t<T^{\bfeta}, \quad \bar\bP-a.s.
		\end{align}	
		 where for a given $\delta=(\delta_1,\delta_2)$,
			$$	T^{\bfeta}  := T\wedge\inf\{t>0:\inf_{\sO}(J_{\bar\bfeta}(t,z,r))\leq \delta_1 \text{ or } \|\bar{\bfeta}(t)\|_{\bH^s(0,L)}\geq \frac1{\delta_2} \}.$$
		Indeed, to show that this is true, let us introduce the following stopping times.
		For $\frac32<s<2$ we define
		\begin{align*}
			T^{\bfeta}_N &:=T\wedge \inf\{t> 0:\inf_{\sO}( J_{\bar\bfeta_N}(t,z,r))\leq \delta_1 \text{ or } \|\bar{\bfeta}_N(t)\|_{\bH^s(0,L)}\geq \frac1{\delta_2}\}.
		\end{align*}
		Then thanks to  \eqref{etaunif}, $T^{\bfeta} \leq \liminf_{N\rightarrow\infty}T^{\bfeta}_N$ a.s. 
		Observe further that for almost any $\omega\in\bar\Omega$ and $t<T^{\bfeta}$, and for any $\epsilon>0$, there exists an $N$ such that
		\begin{align*}
			\|\bar{\bfeta}(t)-\bar{\bfeta}^*(t)\|_{\bH^s(0,L)}&<\|\bar{\bfeta}(t)-\bar{\bfeta}_N(t)\|_{\bH^s(0,L)}+\|\bar{\bfeta}^*_N(t)-\bar{\bfeta}_N(t)\|_{\bH^s(0,L)}+\|\bar{\bfeta}^*(t)-\bar{\bfeta}^*_N(t)\|_{\bH^s(0,L)}\\
			&<\epsilon.
		\end{align*}
		This is true because for any $\epsilon>0$ there exists an $N_1 \in \mathbb{N}$ such that the first and the last terms on the right side of the above inequality
		are each bounded by $\frac{\epsilon}2$ for any $N\geq N_1$ thanks to the uniform convergence \eqref{etaunif}.
		Furthermore, since $t<T^{\bfeta}_N$, for infinitely many $N$'s, the second term is equal to 0. Hence we conclude that indeed
		\begin{align}
			\bar{\bfeta}^*(t)=\bar{\bfeta}(t) \quad \text{ for any } t<T^{\bfeta}, \quad \bar\bP-a.s.
		\end{align}

	\section{Passing to the limit $\ep\to 0$}
		In this section, to emphasize the dependence on the parameter $\ep>0$, we will use the notation $(\bar \bu_\ep,\bar \bv_\ep,\bar \bv^*_\ep, \bar \bfeta_\ep,\bar\bfeta^*_\ep,\bar W_\ep)$ and $(\Omega^\ep,\sF^\ep,(\sF^\ep_t)_{t\geq 0},\bP^\ep)$ to denote the solution and the filtered probability space found in the previous section.  
	In what follows, we will pass $\ep \rightarrow 0$ in \eqref{martingale1} with appropriate test functions. Most of the results in the first half of this section can be proven as in the previous section. Hence we only summarize the important theorems without proof here.
	
	First observe that, thanks to the weak lower-semicontinuity of norm, the uniform estimates obtained in the previous section still hold. That is, as a consequence of Lemmas \ref{bounds} and Theorem \ref{skorohod}, we have the following {uniform boundedness} result. 
	\begin{lem}[Uniform boundedness]\label{boundsep}
		For a fixed $\delta=(\delta_1,\delta_2)$ satisfying \eqref{delta}, we have for some $C>0$ {\bf independent of} $\ep$ that
		\begin{enumerate}
			\item $\bE^\ep\|\bar\bu_\ep\|^2_{L^\infty(0,T;\bL^2(\sO))\cap L^2(0,T;V)}<C.$
			\item $\bE^\ep\|\bar \bv_\ep\|^2_{L^\infty(0,T;\bL^2(0,L))\cap L^2(0,T;\bH^\frac12(0,L))}<C$ and thus $\bE^\ep\|\bar \bv_\ep\|^2_{ L^2(0,T;\bL^p(0,L))}<C$, $\forall p<\infty$.
			\item $\bE^\ep\|\bar \bv^*_\ep\|^2_{L^\infty(0,T;\bL^2(0,L))}<C$.
			\item $\bE^\ep\|\bar \bfeta_\ep\|^2_{L^\infty(0,T;\bH^2_0(0,L)\,\cap\, W^{1,\infty}(0,T;\bL^2(0,L)))}<C.$
			\item $\bE^\ep\|\bar \bfeta^*_\ep\|^2_{L^\infty(0,T;\bH^2_0(0,L)\,\cap\, W^{1,\infty}(0,T;\bL^2(0,L)))}<C.$
			\item $\bE^\ep\|\text{div}^{\bar\bfeta^*_\ep}\bar\bu_\ep\|^2_{L^2(0,T;L^2(\sO))} < C{\ep}$.
			\item $\bE^\ep\|\sqrt{\ep} \partial_z\bar\bv_\ep\|^2_{L^2(0,T;\bL^2(0,L))} < C$.
		\end{enumerate}
	\end{lem}
	Next we state the tightness results:
		\begin{lem}[Tightness of the laws]\label{tightep}
		\begin{enumerate}
			\item The laws of $\bar\bu_\ep$ and $\bar \bv_\ep$ are tight in $L^2(0,T;\bH^\alpha(\sO))$ and $ L^2(0,T;\bH^{\beta}(0,L))$ respectively for any $0\leq \alpha<1$ and $0\leq \beta<\frac12$.
			\item The laws of $\bar \bfeta_\ep$ and that of $\bar \bfeta^*_\ep$ are tight in $C([0,T];\bH^s(0,L))$ for $\frac32<s<2$.
			\item The laws of $\|\bar \bu_\ep\|_{L^2(0,T;V)}$ are tight in $\mathbb{R}$.
			\item The laws of $\|\bar \bv^*_\ep\|_{L^2(0,T;\bL^2(0,L))}$ are tight in $\mathbb{R}$.
		\end{enumerate}
	\end{lem}
		\begin{proof}
		The proof of the first statement follows as the proof of Lemma \ref{tightuv} almost identically. Construction of suitable test functions is the same as \eqref{Qn} and we apply a variant of It\^o's formula as stated in 
		Lemma 5.1 in \cite{BO13}, to "test" \eqref{martingale1} with the continuous-in-time versions of the test functions \eqref{Qn} (i.e. where $(\Delta t)\sum_{n-j+1}^n$ is replaced by $\int_{t-h}^t dt$). Recall that all the bounds obtained in the proof of Lemmas \ref{tightuv}, except in \eqref{ep2}, are independent of $\ep$. However, observe that due to integrating by parts in \eqref{byparts} and applying \eqref{Jform}, we do not have the aforementioned term containing the derivatives of $\bar\bw^*$ in our weak formulation \eqref{martingale1} anymore; and instead we have the boundary term $\int_0^t\int_\Gamma S_{\bar\bfeta_\ep^*}(\bar\bv_\ep^*\cdot\bn_\ep^*)(\bar\bv_\ep\cdot\bfpsi)$. Then, for the test function $\bfpsi_\ep$, which is a suitable regularized approximation of $\bar\bv_\ep$ described above and constructed as in \eqref{Qn}, we can find $\ep-$ independent bounds for this boundary integral by using the fact that $\|\bar\bn^*_\ep\|_{\bL^\infty((0,T)\times(0,L))}<C(\delta)$ and the bounds Lemma \ref{boundsep}(2),(3) for $\bar\bv_\ep$ and $\bar\bv^*_\ep$ respectively.
	\end{proof}
		Now for an infinite denumerable set of indices $\Lambda$, we let 
	$\mu_{\ep}$ be the law of the random variable $\bar{\mathcal{U}}_{\ep}:=(
	\bar{\bu}_{\ep},\bar \bv_{\ep},\bar\bfeta_{\ep}, \bar\bfeta^*_{\ep}, 
	\|\bar \bu_\ep\|_{L^2(0,T;V)},\|\bar \bv^{*}_{\ep}\|_{L^2(0,T;\bL^2(0,L))},
	\bar{W}_\ep)$
	taking values in the phase space
	\begin{align*}
		\sS:=L^2(0,T;\bH^{\frac34}(\sO))\times L^2(0,T; \bH^{\frac14}(0,L)) \times[ C([0,T],\bH^s(0,L))]^2\times \mathbb{R}^2\times
		C([0,T];U),
	\end{align*}
	for $\frac32<s<2$.
	
	Then tightness of $\mu_\ep$ on $\sS$ and an application of the Prohorov theorem and the Skorohod representation theorem gives us the following almost sure representation and convergence.
	\begin{theorem}[Almost sure convergence in $\ep$]\label{skorohod2} There exists a probability space $(\hat\Omega,\hat\sF,\hat\bP)$ and random variables $\hat{\mathcal{U}}_{\ep}=(
		\hat{\bu}_{\ep},\hat \bv_{\ep},\hat\bfeta_{\ep}, \hat\bfeta^*_{\ep}, m_\ep,k_\ep,
		\hat{W}_\ep)$ and
		$\hat{\mathcal{U}}=(
		\hat{\bu},\hat \bv,\hat\bfeta, \hat\bfeta^*,m,k,
		\hat{W})$
		such that
		\begin{enumerate}
			\item $\hat{\mathcal{U}}_{\ep}=^d\bar{\mathcal{U}}_{\ep}$ for every $\ep \in \Lambda$.
			\item $\hat{\mathcal{U}}_{\ep} \rightarrow \hat{\mathcal{U}}$ $\hat\bP$-a.s. in the topology of $\sS$ as $\ep\rightarrow 0$.
			\item $\partial_t\hat\bfeta=\hat \bv$ and $\partial_t\hat\bfeta^*=\hat \bv^*$, in the sense of distributions, almost surely.
		\end{enumerate}
	\end{theorem}
	
	To pass to the limit as $\ep \to 0$ we will need stronger convergence of the fluid velocity random variables to compensate for the fact that our construction of	 test functions presented below does not lead to uniform convergence of the test functions as $\ep\to 0$. See \eqref{convq}-\eqref{convqt}.		For this purpose we start by recalling 
	again that Theorem 1.10.4 in \cite{VW96} implies that the random variables $\hat{\mathcal{U}}_\ep$ can be chosen such that for every $\ep \in \Lambda$,
	\begin{align}\label{newrv2}
		\hat{\mathcal{U}}_\ep(\omega)=\bar{\mathcal{U}}_\ep(\phi_\ep(\omega)), \quad \omega \in \hat\Omega,
	\end{align} and $\hat\bP\circ\phi_\ep^{-1}=\bP^\ep$, where $\phi_\ep:\hat\Omega\rightarrow \Omega^\ep$ is measurable. 
	
	Thanks to these explicit maps we identify the real-valued random variables $m_\ep$ as $m_\ep= \|\hat \bu_\ep\|_{L^2(0,T;V)}$ and notice that $m_\ep$ converge almost surely due to  Theorem~\ref{skorohod2}.}
	The almost sure convergence of $m_\ep$ implies that $\|\hat \bu_\ep\|_{L^2(0,T;V)}$ is bounded a.s. and thus also that, up to a subsequence,
	\begin{align}\label{uweak2}
	\hat \bu_\ep \rightharpoonup \hat\bu \quad\text{ weakly in } {L^2(0,T;V)} \quad \hat\bP-a.s.
	\end{align}
	Similarly,
	\begin{align}\label{vweak2}
	\hat \bv^*_\ep \rightharpoonup \hat \bv^* \quad\text{ weakly in } {L^2(0,T;\bL^2(0,L))} \quad \hat\bP-a.s.
	\end{align}
As in the previous section, we also have
\begin{equation}\label{convrest1}
	\begin{split}
		&  A_{\hat\bfeta^*_\ep} \rightarrow  A_{\hat\bfeta^*} \text{ and } ( A_{\hat\bfeta^*_\ep})^{-1} \rightarrow ( A_{\hat\bfeta^*})^{-1} \quad \text{ in } L^\infty(0,T;\bW^{2,p}(\sO))\, a.s. \text{ for any } p<4.\\
		&			 J_{\hat\bfeta^*_\ep} \rightarrow  J_{\hat\bfeta^*} =\text{det}(\nabla A_{\hat\bfeta^*})\quad \text{ in } L^\infty(0,T;C(\hat\sO))\, a.s.
		\\	&			\hat S_{\hat\bfeta^*_\ep} \rightarrow  S_{\hat\bfeta}\quad \text{ in } L^\infty(0,T;C(\hat\Gamma)) \, a.s.\\
	&	\hat\bw^*_\ep \rightharpoonup \hat\bw^*\, \quad \text{ weakly in } L^2(0,T;\bH^{\frac12}(\sO)),\quad \, \hat\bP-a.s,\\
		&\hat	\bn^*_\ep \to \hat\bn^*\, \quad \text{  in } L^\infty(0,T;C(\bar\Gamma)),\quad \, \hat\bP-a.s.
	\end{split}
\end{equation}
Before we pass to the limit $\ep\to 0$, 
	we have one more obstacle to deal with. Namely,
	that the candidate solution for fluid and structure velocities, $\hat\bU=(\hat\bu,\hat \bv)$, is not regular enough in time to be a stochastic process in the classical sense as it only belongs to the space $L^2(0,T;\bL^2(\sO))\times L^2(0,T;\bL^2(0,L))$ (note that the equivalent of Lemma \ref{uvtightC} does not apply). Hence  we construct an appropriate filtration as follows. 
	Define the $\sigma-$fields
	$$
	\sigma_t(\hat\bU):=\bigcap_{s\geq t}\sigma\left(\bigcup_{\bQ\in C^\infty_0((0,s);\sD)}\{(\hat\bU,\bQ)<1\}\cup \mathcal{N} \right),\quad \mathcal{N}=\{\mathcal{A}\in \hat{\mathcal{F}} \ |  \hat\bP(\mathcal{A})=0\}.
	$$
	Let $\hat\sF_t'$ be the $\sigma-$ field generated by the random variables $\hat\bfeta(s),\hat{W}(s)$ for all $0\leq s \leq t$.
	Then we define
	\begin{align}\label{Ft1}
		\hat{\mathcal{F}}^0_t:=\bigcap_{s\ge t}\sigma(\hat{\sF}_s' \cup \mathcal{N}),\qquad
		\hat{\mathcal{F}}_t :=\sigma(\sigma_t(\hat\bU)\cup \hat{\mathcal{F}}^0_t).
	\end{align}
	This gives a complete, right-continuous filtration $(\hat{\mathcal{F}}_t)_{t \geq 0}$, on the probability space $(\hat\Omega,\hat{\mathcal{F}},\hat\bP)$, to which the noise processes and candidate solutions are adapted, see \cite{BFH18}.
	\begin{lem}\label{representative}
		There exists a stochastic process in $
		L^2(0,T;\bL^2(\sO))\times L^2(0,T; \bL^2(0,L))$ a.s. which is a $(\hat\sF_t)_{t\geq 0}$-progressively measurable representative of $\hat\bU=(\hat\bu, \hat \bv)$.
	\end{lem}
	
		\begin{theorem}\label{exist2} {\s\bf (Main result.)}		
		For any fixed $\delta$, for which \eqref{delta} is satisfied, {\s the stochastic processes $(\hat \bu,\hat {\bfeta},\hat{\bfeta}^*)$ constructed in Theorem~\ref{skorohod2}}
		satisfy
		\begin{equation}\begin{split}\label{martingale2}
				&(	{J}_{\hat\bfeta^*}(t)\hat\bu (t),\bq(t))+(\partial_t\hat{ {\bfeta}}(t),\bfpsi(t))
				=((J_0)\bu_0,\bq(0))+( \bv_0,\bfpsi(0))+\int_0^t\int_0^L \partial_t\hat {\bfeta} \,\partial_t{\bfpsi}\\
				&+\int_0^t\int_\sO{ J}_{\hat\bfeta^*}\hat\bu\cdot \partial_t{\bq} -\int_0^t \langle \sL_e(\hat{\bfeta}),\bfpsi \rangle -2\nu\int_0^t\int_{\sO}{J}_{\hat\bfeta^*}\, \bD^{\hat{\bfeta} ^*}(\hat\bu )\cdot \bD^{\hat{\bfeta} ^*}(\bq) \\
					&-\frac12\int_0^t\int_{\sO}J_{\hat\bfeta^*}(\hat\bu\cdot\nabla^{\hat\bfeta ^*}\hat\bu \cdot\bq
				- (\hat\bu^{}-2\hat\bw^*)\cdot\nabla^{\hat\bfeta ^*}\bq\cdot\hat\bu ) 		+\frac12\int_0^t\int_\Gamma S_{\hat\bfeta^*}(\hat\bv^*\cdot\hat\bn^*)(\hat\bv\cdot\bfpsi)\\
				&+ 				\int_0^t\left( P_{{in}}\int_{0}^1q_z\Big|_{z=0}dr-P_{{out}}\int_{0}^1q_z\Big|_{z=1}dr\right)				ds 
			+\int_0^t
				(	G(\hat\bu,\hat \bv ,\hat{\bfeta}^*)d\hat W , \bQ)  ,
			\end{split}
		\end{equation}
		$\hat\bP$-a.s. for almost every $t\in[0,T]$ and for any $(\hat\sF_t)_{t \geq 0}-$adapted process $\bQ=(\bq,\bfpsi)$ with continuous paths in $\sD$ such that $\nabla^{\hat{\bfeta}^*}\cdot\bq=0$ and $\bq|_{\Gamma}=\bfpsi$ a.s. 
	\end{theorem}
	\begin{proof}[Proof of Theorem \ref{exist2}]
We will begin the proof by giving a  construction for $\sD$-valued processes $(\bq_{\ep},\bfpsi_{\ep})$, satisfying appropriate boundary conditions on $\Gamma$ and such that $\bq_\ep$ satisfies the divergence-free condition. This is required so that the penalty terms in the fluid sub-problem drop out. 

We begin by constructing an appropriate test functions for the limiting equation \eqref{martingale2} as follows: Recall that the maximal domain $\sO_\delta=(0,L)\times(0,R_\delta)$ is a rectangular domain comprising of all the moving domains $\sO_{\hat\bfeta_\ep^*}$. Consider a smooth $(\hat\sF_t)_{t\geq0}$-adapted process $\bg=(g_z,g_r)$ on $\bar\sO_{\delta}$ such that $\nabla\cdot \bg=0$ and such that $\bg$ satisfies the required boundary conditions  $g_r=0 \text{ on } z=0,L, r=0$  and $\partial_r g_z=0 \text{ on }\Gamma_{b}$. Assume also that, for some $\bfpsi=(\psi_z,\psi_r)$, on the top lateral boundary of the moving domain associated with $\hat{\bfeta}^*$, $\Gamma_{\hat{\bfeta}^*}$, the function $\bg$ satisfies
$\bg(t)|_{\Gamma_{\hat{\bfeta}^*}(t)}=\bfpsi(t)$.
Next, we define
$$ \bq(t,z,r,\omega) = \bg(t,\omega) \circ\, A^\omega_{\hat{\bfeta}^*}(t)(z,r) .$$	
Now, to observe that $\bq$ is a suitable test function we define, for any $t\in[0,T]$ and given process $\bg$ as mentioned above, $\sC_\bg:\hat\Omega\times \bC([0,L]) \rightarrow \bC^1(\bar\sO)$ as $$\sC_{\bg}(\omega,{\bfeta})
=F_{\bfeta}(\bg(t,\omega)),$$
where $F_{\bfeta}({\bf f}):={\bf f}\circ A^\omega_{{\bfeta}}$ is a well-defined map from $\bC(\bar\sO_{{\bfeta}})$ to $\bC(\bar\sO)$ for  any ${\bfeta}\in \bC([0,L])$. Thanks to the continuity of the composition operator $F_{\bfeta}$ and the assumption that $\bg(t)$ is $\hat\sF_t-$measurable, we obtain that, for any ${\bfeta}$, the $\bC^1(\bar\sO)$-valued map $\omega \mapsto \sC_\bg(\omega,{\bfeta})$ is $\hat\sF_t$-measurable (where $\bC^1(\bar\sO)$ is endowed with Borel $\sigma$-algebra).
Note also that for any fixed $\omega$, the map  ${\bfeta}\mapsto \sC_\bg(\omega,{\bfeta})$ is continuous.
Hence we deduce that $\sC_\bg$ is a Carath\'eodory function.
Now by the construction of the filtration $(\hat\sF_t)_{t\geq 0}$ in \eqref{Ft1}, we know that $\hat{\bfeta}^*$, is $(\hat\sF_t)_{t\geq 0}$-adapted. Therefore, we conclude that the $\bC^1(\bar\sO)$-valued process $\bq$, which by definition is $\bq(t,\omega)=\sC_\bg(\omega,\hat{\bfeta}^*(t,\omega))$ is $(\hat\sF_t)_{t\geq 0}$-adapted as well. The same conclusions for the process $\bfpsi$ follow using the same argument.

We summarize here that the $\{\hat\sF_t\}_{t \geq 0}-$adapted processes $(\bq,\bfpsi)$ have continuous paths in $\sD$ such that $$\nabla^{\hat\bfeta^*}\cdot\bq=0 \text{ and } \bq\Big|_{\Gamma}=\bfpsi.$$
We have that for any $\omega\in\hat\Omega$ that $\bq\in L^\infty(0,T;\bH^{2}(\sO)) \cap H^1(0,T;\bH^1(\sO))$.
Now we define the approximate test functions $(\bq_{\ep},\bfpsi_{\ep})$, using the Piola transformation as in the proof of Lemma \ref{tightuv}, as follows:
\begin{align*}
	&\bq_{\ep}={J^{-1}_{\hat\bfeta^*_\ep}}\nabla A_{\hat\bfeta^*_{\ep}}{J^{}_{\hat\bfeta^*}} \nabla  A^{-1}_{\hat\bfeta^*} \left( \bq-\bfpsi\chi\right) +\bfpsi\chi-\frac{\lambda^{\hat{\bfeta}^*_{\ep}}-\lambda^{\hat\bfeta^*}}{\lambda^{\ep}_0}(\xi_0\chi)\\
	&+{J^{-1}_{\hat\bfeta^*_\ep}}\sB\left( \text{div}\left( (J_{\hat\bfeta^*}(\nabla A_{\hat\bfeta^*})^{-1}-J_{\hat\bfeta_\ep^*}(\nabla A_{\hat\bfeta_\ep^*})^{-1})\bfpsi\chi-\frac{\lambda^{\hat{\bfeta}^*_{\ep}}-\lambda^{\hat\bfeta^*}}{\lambda^{\ep}_0}J_{\hat\bfeta_\ep^*}(\nabla A_{\hat\bfeta_\ep^*})^{-1}\xi_0\chi\right) \right),
\end{align*}	
and,
\begin{align*}	
	&	\bfpsi_{\ep}= \bfpsi-\frac{\lambda^{\hat{\bfeta}^*_{\ep}}-\lambda^{\hat{\bfeta}^*}}{\lambda^{\ep}_0}\xi_0,
\end{align*}	
where we pick an appropriate $\xi_0 \in \bC_0^\infty(\Gamma)$ such that $\lambda_0^\ep$ defined below is not 0 for any $\ep>0$, 
$$\lambda_0^{\ep}(t)=-\int_\Gamma ({\bf id}+\hat{\bfeta}^*_{\ep}(t))\times \partial_z\xi_0\,dz \quad\text{and},$$
$$\lambda^{\hat{\bfeta}^*_{\ep}}(t)=-\int_\Gamma ({\bf id} +\hat{\bfeta}_{\ep}^*(t))\times \partial_z\bfpsi (t)dz,\quad \lambda^{\hat{\bfeta}^*_{}}(t)=-\int_\Gamma ({\bf id} +\hat{\bfeta}_{}^*(t))\times \partial_z\bfpsi (t)dz.$$
Here $\chi(r)$ is a smooth function on $\sO$ such that $\chi(1)=1$ and $\chi(0)=0$.
Finally, $\sB:L^2(\sO) \to \bH^1_0(\sO)$ is Bogovski's operator (see \cite{Galdi}) which is used to correct the divergence of the extra terms appearing in the definition of $\bq_\ep$ due to the extension of structure velocities in the fluid domains. Note that, $\sB$ can be continuously extended to a bounded operator from $W^{l,p}_0(\sO)$ to $\bW^{l+1,p}_0(\sO)$ for any $l>-2+\frac1p$ (see \cite{GHH06}, \cite{Galdi}).

Now observe that, thanks to the properties of the Piola transformation (see e.g. Theorem 1.7 in \cite{C88}), we have
$$\nabla^{\hat\bfeta^*_\ep} \cdot\bq_{\ep}=J^{}_{\hat\bfeta^*}J^{-1}_{\hat\bfeta_\ep^*} \nabla^{\hat\bfeta^*}\cdot\bq=0, \qquad \text{and}\qquad \bq_\ep|_\Gamma=\bfpsi_\ep.$$
Observe that thanks to Theorem \ref{skorohod2}, we have that $\lambda^{\hat{\bfeta}^*_{\ep}}\to \lambda^{\hat{\bfeta}^*}$ in $L^\infty(0,T)$ a.s. 
Additionally, thanks to \eqref{convrest1} 
we obtain
\begin{align}
\notag	\|	\bq_\ep-\bq\|_{L^\infty(0,T;\bH^1(\sO))} &\leq \|A_{\hat\bfeta^*_\ep}-A_{\hat\bfeta^*}\|_{L^\infty(0,T;\bW^{2,3}(\sO))}\left( \|\bq\|_{L^\infty(0,T;\bH^1(\sO))}+\|\bfpsi\|_{L^\infty(0,T;\bH^1(0,L))}\right) \\
	&+\|\lambda^{\hat{\bfeta}^*_{\ep}}-\lambda^{\hat{\bfeta}^*}\|_{L^\infty(0,T)}\to 0 \qquad \hat\bP-\text{ a.s.}\label{convq}
\end{align}
Observe also that
$$|\frac{d}{dt}\lambda^{\hat{\bfeta}^*_{\ep}}|
\leq \|\hat\bv^*_\ep\|_{\bL^{2}(0,L)}\|\bfpsi\|_{\bH^{1}(0,L)}+\|\hat\bfeta^*_\ep\|_{\bL^\infty(0,L)}\|\partial^2_{tz}\bfpsi\|_{\bL^\infty(0,L)}.$$
Hence $\lambda^{\hat{\bfeta}^*_{\ep}}\to \lambda^{\hat{\bfeta}^*}$ weakly 
in $H^1(0,T)$ a.s. 
Furthermore, we have (see e.g. \cite{BH21}), for
 $1<m<\frac32$, that
\begin{align*}
	\|\partial_t\bq_\ep\|_{\bH^{-\frac12}(\sO)} &\leq C(\delta)\|\bw^*_\ep\|_{\bH^{\frac12}(\sO)}\| \nabla A_{\hat\bfeta^*}\|_{\bH^{m}(\sO)}\|\bq\|_{\bH^m(\sO)} + C(\delta)\|\partial_t\bq\|_{\bL^2(\sO)}\\
	&+\| \lambda^{\hat\bfeta^*_\ep}-\lambda^{\hat\bfeta^*}\|_{H^1(0,T)}\leq C.
\end{align*}
Hence, we obtain that,
\begin{align}\label{convqt}
	\partial_t\bq_\ep \rightharpoonup \partial_t\bq
	\quad\text{ weakly in } L^2(0,T;\bH^{-\alpha}(\sO)), \quad \text{ for any } 1/2\leq \alpha\leq 1\quad \hat\bP-a.s. \, 
\end{align}
Similarly, for any $k$ we have that 
\begin{equation}\begin{split}\label{convpsi}
		\bfpsi_{\ep} \rightarrow \bfpsi \quad \text{ in }L^\infty(0,T;\bC^k(\bar\Gamma)),\, \quad \hat\bP-a.s.\\
		\partial_t\bfpsi_{\ep} \rightharpoonup \partial_t \bfpsi \quad \text{ weakly in }L^2(0,T;\bC^k(\Gamma)),\, \quad \hat\bP-a.s.
	\end{split}
\end{equation}
			We are now in a position to take the limit as $\ep\to 0$ in the weak formulation \eqref{martingale2}. Namely, we consider \eqref{martingale2}  
			and use $\bQ_{\ep}=(\bq_{\ep},\bfpsi_\ep)$ as the test functions.
			This, as mentioned earlier, requires a special version of the It\^o formula given in Lemma 5.1 in \cite{BO13}. We obtain that
		\begin{equation}\begin{split}\label{epeq}
				&(	(\hat{  J}_\ep^*(t))\hat{ \bu}_\ep(t), \bq_\ep(t))+(\hat{  \bv}_\ep(t), \bfpsi_\ep(t))
				=(J_0\bu_0,\bq(0))+(\bv_0,\bfpsi(0))\\
				&-\int_0^t\langle \sL_e(\hat{\bfeta}_\ep),\bfpsi_\ep\rangle -\ep\int_0^t\int_0^L\partial_z\hat\bv\cdot\partial_z\bfpsi 
				+\int_0^t\int_{\sO}  {\hat J}_\ep^*\hat{ \bu}_\ep\cdot {\partial_t  \bq_\ep}+\int_0^t\int_0^L {\partial_t { \bfpsi}_\ep}\cdot\hat\bv_\ep\\
						&-\frac12\int_0^t\int_{\sO}{\hat J}_\ep^*(\hat\bu_\ep\cdot\nabla^{\hat{\bfeta}_\ep^*}\hat\bu^{}_\ep\cdot\bq
				_\ep- (\hat\bu_\ep-
				2\hat\bw^*_\ep)\cdot\nabla^{\hat{\bfeta}_\ep^*}\bq_\ep\cdot\hat\bu^{}_\ep)\\
					&+\frac12\int_0^t\int_\Gamma S_{\hat\bfeta^*_\ep}(\hat\bv^*_\ep\cdot\hat\bn^*_\ep)(\hat\bv_\ep\cdot\bfpsi_\ep)-2\nu\int_0^t\int_{\sO}({\hat J}_\ep^*)\bD^{\hat{\bfeta}_\ep^*}(\hat\bu_\ep)\cdot \bD^{\hat{\bfeta}_\ep^*}(\bq_\ep) \\
				&+ 				\int_0^t\left( P_{{in}}\int_{0}^1(q_\ep)_z\Big|_{z=0}-P_{{out}}\int_{0}^1(q_\ep)_z\Big|_{z=1}\right)
				+\int_0^t
				 (G({\hat \bu}_\ep,\hat\bv_\ep,\hat{\bfeta}^*_\ep)d\hat W_\ep. \bQ_\ep),.
		\end{split}\end{equation}
			We can now pass $\ep\to 0$ in the deterministic terms in \eqref{epeq} using the convergence results obtained in Theorem \ref{skorohod2}, \eqref{convrest1}, \eqref{convq}, \eqref{convqt}, \eqref{convpsi} (see \cite{TC23}). We only demonstrate the passage of $\ep\to 0$ in the stochastic term.		
		Observe that, by construction, $\bQ$ is $(\hat\sF_t)_{t\geq 0}$-adapted.
		\begin{lem}\label{conv_G}
			The processes $
			\left( \int_0^t(G(\hat{\bu}_{\ep}(s),\hat\bv_\ep(s),\hat{\bfeta}^*_{\ep}(s))d\hat{W}_{\ep}(s),\bQ_{\ep}(s))\right) _{t \in [0,T]}$ converge to the processes $\left( \int_0^t(G(\hat{\bu}(s),\hat\bv(s),\hat{\bfeta}^*(s))d\hat{W}(s),\bQ(s))\right) _{t \in [0,T]}$  in $L^1(\hat{ \Omega};L^1(0,T;\mathbb{R}
			))$ as $ \ep\to 0$.
		\end{lem}
		\begin{proof}
			First, by using \eqref{growthG} we observe that for $\frac32<s<2$ we have
			\begin{align*}
				&	\int_0^T\|(G(\hat{\bu}_{\ep},\hat\bv_\ep,\hat{\bfeta}^*_{\ep}),\bQ_{\ep})-( G(\hat \bu,\hat\bv,\hat{\bfeta}^*),\bQ)\|^2_{L_2(
					U_0,\R)}d s\\
				&\leq \int_0^T \|(G(\hat{\bu}_{\ep},\hat\bv_\ep,\hat{{\bfeta}}^*_{\ep})-G(\hat{\bu},\hat\bv,\hat{{\bfeta}}^*_{\ep}),\bQ_{\ep})\|^2_{L_2(U_0;\mathbb{R})} + \int_0^T\|(G(\hat{\bu},\hat\bv,\hat{{\bfeta}}^*_{\ep}),\bQ_{\ep}-\bQ)\|^2_{L_2(U_0;\mathbb{R})}\\
				&\qquad +\int_0^T\|(G(\hat\bu,\hat\bv,\hat{\bfeta}^*_{\ep})-G(\hat\bu,\hat\bv,\hat{\bfeta}^*),\bQ)\|^2_{L_2(U_0;\mathbb{R})}\\
				&\le \int_0^T\left( \|\hat{\bfeta}^*_{\ep}\|^2_{\bH^s(0,L)}\|\hat{\bu}_{\ep}-\hat\bu\|^2_{\bL^2(\sO)}+\|\hat\bv_{\ep}-\hat\bv\|_{\bL^2(0,L)}^2\right) \|\bQ_{\ep}\|_{\bL^2}^2\\
				&+  \int_0^T\left( \|\hat{\bfeta}^*_{\ep}\|^2_{\bH^s(0,L)}\|\hat\bu\|^2_{\bL^2(\sO)}+\|\hat\bv\|_{\bL^2(0,L)}^2\right) \|\bQ_{\ep}-\bQ\|_{\bL^2}^2+ \int_0^T\|\hat{\bfeta}^*_{\ep}-\hat{\bfeta}^*\|^2_{\bL^\infty(0,L)}\|\hat{\bu}\|^2_{\bL^2(\sO)}\|\bQ\|_{\bL^2}^2\\
				&\leq \frac1{\delta}\|\bQ_{\ep}\|^2_{L^2(0,T;\bL^2)}\|\hat\bu_{\ep}-\hat\bu\|^2_{L^2(0,T;\bL^2(\sO))}\\
				&+\frac1{\delta}\|\bQ_{\ep}-\bQ\|^2_{L^2(0,T;\bL^2)}\left( \|\hat{\bu}\|_{L^2(0,T;\bL^2(\sO))}^2 +\|\hat\bv\|^2_{L^2(0,T;\bL^2(0,L))}\right) \\
				&+ \|\hat{\bfeta}^*_{\ep}-\hat{\bfeta}^*\|_{L^\infty(0,T;\bH^s(0,L))}\|\hat{\bu}\|^2_{L^2(0,T;\bL^2(\sO))}\|\bQ\|_{L^2(0,T;\bL^2)}^2.
			\end{align*}	
			Thanks to Theorem \ref{skorohod2}, \eqref{convq} and \eqref{convpsi} the right hand side of the inequality above converges to 0, $\hat\bP$-a.s. as $\ep \to 0$.	Hence we have proven that 
			\begin{align}\label{g1}
				(G(\hat{\bu}_{\ep},\hat\bv_\ep,\hat{\bfeta}^*_{\ep}),\bQ_{\ep})\rightarrow( G(\hat \bu,\hat\bv,\hat{\bfeta}^*),\bQ), \qquad \text{$\hat \bP$-a.s \quad in $L^2(0,T;L_2(U_0,\R		)).$ }\end{align}
			Now using classical ideas from \cite{Ben} (see also Lemma 2.1 of \cite{DGHT} for a proof), the convergence \eqref{g1} implies that
			\begin{align}\label{G2}
				\int_0^t(G(\hat{\bu}_{\ep},\hat\bv_\ep,\hat{\bfeta}^*_{\ep})d\hat W_\ep,\bQ_{\ep}) \rightarrow \int_0^t( G(\hat \bu,\hat\bv,\hat{\bfeta}^*)d\hat{W},\bQ)
			\end{align}
			in probability in $L^2(0,T;\R)$.
			
			Furthermore observe that for some $C>0$ independent of $\ep$ we have the following bounds which follow from the It\^{o} isometry:
			\begin{align}
				\hat \bE\int_0^T |\int_0^t (G(\hat{\bu}_{\ep},\hat\bv_\ep,\hat{\bfeta}^*_{\ep})d\hat W_\ep(s),\bQ_{\ep})&|^2d t=\int_0^T\hat \bE\int_0^t\|(G(\hat{\bu}_{\ep},\hat\bv_\ep,\hat{\bfeta}^*_{\ep}),\bQ_{\ep})\|^2_{L_2(U_0,\R)}d s d t \notag\\
				& \le  T
				\hat \bE\left(\int_0^T\left( \|{\hat{\bfeta}^*_{\ep}}\|_{\bH^s(0,L)}^2\|\hat{\bu}_{\ep}\|^2_{\bL^2(\sO)}+\|\hat\bv_{\ep}\|^2_{\bL^2(0,L)}\right) d s\right)\label{G3}\\
				&\le   C(\delta).\notag
			\end{align}
			Here we also used the fact that $\|\bq_\ep\|_{L^\infty(0,T;\bL^2(\sO))} \leq C(\delta)(\|\bq\|_{L^\infty(0,T;\bL^2(\sO))}+\|\bfpsi\|_{L^\infty(0,T;\bL^2(0,L))})$ a.s.
			Combining \eqref{G2}, \eqref{G3} and using the Vitali convergence theorem, we thus conclude the proof of Lemma \ref{conv_G}.
			
		\end{proof}			
	
		
	\end{proof}
	
	{\s Notice that in the statement of Theorem~\ref{exist2} we still have the function $\hat{\bfeta}^*(t)$ in the weak formulation, which keeps the displacement uniformly bounded. 
		We now show that in fact $\hat{\bfeta}^*(t)$ can be replaced by the limiting stochastic process $\hat{\bfeta}(t)$ to obtain the desired weak formulation 
		and martingale solution until some stopping time 
		$T^{\bfeta}$, which we show is strictly greater than zero almost surely.
		
		\begin{lem} \label{StoppingTime}{\bf{(Stopping time.)}}
			Let the deterministic initial data ${\bfeta}_0$ satisfy the assumptions \eqref{etainitial}.
			Then, for any $\delta =(\delta_1,\delta_2)$ satisfying \eqref{delta} and for a given $\frac32<s<2$,
			there exists an almost surely positive stopping time $T^{\bfeta}$, given by
			\begin{equation}\label{stoppingT}
				T^{\bfeta}:=T\wedge\inf\{t>0:\inf_{\sO}  J_{\hat\bfeta}(t)\leq \delta_1\} \wedge\inf\{t>0:\|\hat{\bfeta}(t)\|_{\bH^s(0,L)}\geq \frac1{\delta_2} \},
			\end{equation}
			such that
			\begin{align}\label{etasequal2}
				\hat	{\bfeta}^*(t)=\hat{\bfeta}(t) \quad \text{for } t<T^{\bfeta}.
			\end{align}
		\end{lem}
		
		\begin{proof}
			To prove this lemma we write the stopping time as 
			$$T^{\bfeta} = T\wedge T^{\bfeta}_1+T^{\bfeta}_2, $$
			where $T^{\bfeta}_1$ and $T^{\bfeta}_2$ are defined by:
			$$	
			T^{\bfeta}=T\wedge\inf\{t>0:\inf_{\sO}  J_{\hat\bfeta}(t)\leq \delta_1\} \wedge\inf\{t>0:\|\hat{\bfeta}(t)\|_{\bH^s(0,L)}\geq \frac1{\delta_2} \}=:T\wedge T^{\bfeta}_1+T^{\bfeta}_2.
			$$
			
			We start with $T^{\bfeta}_2$. Observe that using the triangle inequality, for any $\delta_0>{\delta_2}$, we obtain
			\begin{align*}
				\hat\bP[T^{\bfeta}_2=0, \|{\bfeta}_0\|_{\bH^2(0,L)}&<\frac1{\delta_0}] =\lim_{\epsilon\rightarrow 0}\hat\bP[T^{\bfeta}_2<\epsilon,\|{\bfeta}_0\|_{\bH^2(0,L)}<\frac1{\delta_0}]\\
				&\leq \limsup_{\epsilon \rightarrow 0^+}\hat\bP[\sup_{t\in[0,\epsilon)}\|\hat{\bfeta}(t)\|_{\bH^s(0,L)}>\frac1{\delta_2},\|{\bfeta}_0\|_{\bH^2(0,L)}<\frac1{\delta_0}]\\
				&\leq \limsup_{\epsilon \rightarrow 0^+}\hat\bP[\sup_{t\in[0,\epsilon)}\|\hat{\bfeta}(t)-{\bfeta}_0\|_{\bH^s(0,L)}>\frac1{\delta_2}-\frac1{\delta_0}] \\
				&\leq \frac1{(\frac1{\delta_2}-\frac1{\delta_0})}\limsup_{\epsilon \to 0}\hat\bE[\sup_{t\in[0,\epsilon)}\|\hat{\bfeta}(t)-{\bfeta}_0\|_{\bH^s(0,L)}]\\
				&\leq \frac1{(\frac1{\delta_2}-\frac1{\delta_0})}\limsup_{\epsilon\rightarrow 0}\hat\bE[\sup_{t\in[0,\epsilon)}\|\hat{\bfeta}(t)-{\bfeta}_0\|^{1-\frac{s}2}_{\bL^2(0,L)}\|\hat{\bfeta}(t)-{\bfeta}_0\|^{\frac{s}2}_{\bH^2(0,L)}]\\
				&\leq \frac1{(\frac1{\delta_2}-\frac1{\delta_0})}\limsup_{\epsilon\rightarrow 0}\hat\bE[\sup_{t\in[0,\epsilon)}\epsilon\|\hat\bv(t)\|^{1-\frac{s}{2}}_{\bL^2(0,L)}\|\hat{\bfeta}(t)-{\bfeta}_0\|^{\frac{s}{2}}_{\bH^2(0,L)}]\\
				&\leq \limsup_{\epsilon\rightarrow 0} \frac{\epsilon}{(\frac1{\delta_2}-\frac1{\delta_0})}\left( \hat\bE[\sup_{t\in[0,\epsilon)}\|\hat\bv(t)\|^{2}_{\bL^2(0,L)}]\right)^{\frac{2-s}4} \left( \hat\bE[\sup_{t\in(0,\epsilon)}\|\hat{\bfeta}(t)-{\bfeta}_0\|^{2}_{\bH^2(0,L)}]\right)^{\frac{s}4} \\
				&=0.
			\end{align*}
			Hence, by continuity from below, we deduce that for any $\delta_2>0$, 
			\begin{align}
				\hat\bP[T^{\bfeta}_2=0, \|{\bfeta}_0\|_{\bH^2(0,L)}<\frac1{\delta_2}]=0.
			\end{align}
			
			To estimate $T^{\bfeta}_1$ we observe that, similarly, since for any $t\in[0,T]$ we have that $\inf_{\sO}  J_{\hat\bfeta}(t)\geq \inf_{\sO} \hat J_0-\| J_{\hat\bfeta}(t)-J_0\|_{C(\sO)}$, we write for any $\delta_0>\delta_1$ 
			\begin{align*}
				\hat\bP[T^{\bfeta}_1=0, \inf_{\sO}J_0>\delta_0] &\leq \limsup_{\epsilon \rightarrow 0^+}\hat\bP[\inf_{t\in[0,\epsilon)}\inf_{\bar\sO}  J_{\hat\bfeta}(t)<\delta_1,\inf_{\sO} \hat J_0>\delta_0]\\
				&\leq \limsup_{\epsilon \rightarrow 0^+}\hat\bP[\sup_{t\in[0,\epsilon)}\| J_{\hat\bfeta}(t)-J_0\|_{C(\sO)}>\delta_0-\delta_1] \\
				&\leq \frac1{(\delta_0-\delta_1)^2}\limsup_{\epsilon\rightarrow 0}\hat\bE[\sup_{t\in[0,\epsilon)}\| J_{\hat\bfeta}(t)-J_0\|^2_{C(\sO)}]\\
				&=0.
			\end{align*}
			Hence for given $\delta_1>0$, 
			\begin{align}
				\hat\bP[T^{\bfeta}_1=0, \inf_{\sO}J_0>\delta_1]=0.
			\end{align}
			That is we have,
			\begin{align}
				\hat\bP[T^{\bfeta}=0, \inf_{\sO}J_0>\delta_1, \|{\bfeta}_0\|_{\bH^2(0,L)}<\frac1{\delta_2}]=0.
			\end{align}
			
		\end{proof}
	}

	Thus, by combining Theorem \ref{exist2}, Lemma~\ref{StoppingTime}, and  \eqref{etasequal2}, we finish the proof of the main result:
	
			\begin{thm}\label{MainTheorem}{\bf{(Main result.)}}
			For any given $\delta=(\delta_1,\delta_2)$, where $\delta_1>0$ and $\delta_2$ satisfies \eqref{delta}, if the deterministic initial data  ${\bfeta}_0$ satisfies \eqref{etainitial}, then the stochastic processes $(\hat\bu,\hat{\bfeta},T^{\bfeta})$
			obtained in the limit specified in Theorem~\ref{skorohod2},  along with the stochastic basis constructed in Theorem \ref{skorohod2}, 
			determine a martingale solution 
			in the sense of Definition \ref{def:martingale} of the stochastic FSI problem \eqref{u}-\eqref{ic}, with the stochastic forcing given by \eqref{StochasticForcing}.\end{thm}


	\bibliography{stochfsi}
	
	\bibliographystyle{plain}
\end{document}